\documentclass[12pt,reqno,twoside]{amsart}
\usepackage{graphicx}
\UseRawInputEncoding 
\usepackage{subfigure}
\usepackage{amsmath}
\usepackage{multirow}
\usepackage{breqn}
\usepackage{float}
\usepackage{cite}
\usepackage{enumitem}
\usepackage{environ}
\usepackage{mathtools}
\NewEnviron{myequation}{%
	\begin{equation}
		\scalebox{1.1}{$\BODY$}
	\end{equation}
}
\usepackage{color}
\usepackage{soul}
\renewcommand{\baselinestretch}{1.2}
\makeatletter
\newcommand{\single}{\let\CS=\@currsize\renewcommand{\baselinestretch}{1.1}\tiny\CS}
\newcommand{\singb}{\let\CS=\@currsize\renewcommand{\baselinestretch}{1}\tiny\CS}
\newcommand{\singa}{\let\CS=\@currsize\renewcommand{\baselinestretch}{1.2}\tiny\CS}
\newcommand{\oneandahalfspacing}{\let\CS=\@currsize\renewcommand{\baselinestretch}{1.5}\tiny\CS}
\newcommand{\singlespacing}{\let\CS=\@currsize\renewcommand{\baselinestretch}{1.6}\large\CS}
\newcommand{\bc}{\begin{center}}
	\newcommand{\ec}{\end{center}}
\newcommand{\be}{\begin{eqnarray}}
	\newcommand{\ee}{\end{eqnarray}}

\makeatletter
\newcommand{\beano}{\begin{eqnarray*}}
	\newcommand{\eeano}{\end{eqnarray*}}

\newcommand{\ba}{\begin{array}}
	\newcommand{\ea}{\end{array}}

\makeatletter

\newtheorem{cor}{Corollary}[section]

\newtheorem{theorem}{Theorem}[section]

\newtheorem{remark}[theorem]{Remark}
\allowdisplaybreaks
\begin{document}
	
	\title[ Generalized H-fold sumset and Subsequence sum]{Generalized H-fold sumset and Subsequence sum}
	
	
	\author[Mohan]{Mohan}
	\address{Department of Mathematics, Indian Institute of Technology Roorkee, Uttarakhand, 247667, India}
	\email{mohan@ma.iitr.ac.in, mohan98math@gmail.com}

	
	\author[R K Pandey]{Ram Krishna Pandey$^{\dagger}$}
	\address{Department of Mathematics, Indian Institute of Technology Roorkee, Uttarakhand, 247667, India}
	\email{ram.pandey@ma.iitr.ac.in}
	\thanks{$^{\dagger}$The corresponding author}

	
	
	

	\begin{abstract}
		
		Let $A$ and $H$ be  nonempty finite sets of integers and positive integers, respectively. The \textit{generalized $H$-fold sumset}, denoted by $H^{(r)}A$, is the union of the sumsets $h^{(r)}A$ for $h\in H$ where, the sumset $h^{(r)}A$ is the set of all integers that can be represented as a sum of $h$ elements from  $A$ with no summand in the representation appearing more than $r$ times. In this paper, we find the optimal lower  bound for the cardinality of $H^{(r)}A$, i.e., for $|H^{(r)}A|$ and the structure of the  underlying sets $A$ and $H$ when $|H^{(r)}A|$ is equal to the optimal lower bound in the cases  $A$ contains only positive integers and $A$ contains only nonnegative integers. This generalizes recent results of Bhanja. Furthermore, with a  particular set $H$, since $H^{(r)}A$ generalizes \textit{subsequence sum} and hence \textit{subset sum},  we get several results of subsequence sums and subset sums as special cases.
		
	\end{abstract}
	
	\maketitle
	

	\setcounter{page}{1}
	\noindent \section{Introduction}
	
	Let $\mathbb{N}$ be the set of positive integers. Let $A=\lbrace a_{1}, \ldots, a_{k}\rbrace$ be a nonempty finite set of integers and $h$ be a positive integer. The {\it h-fold sumset}, denoted by $hA$, and the {\it restricted h-fold sumset}, denoted by $h^{\wedge}A$ of $A$, are defined, respectively, by
	
	\[hA:=\left\lbrace \sum_{i=1}^{k} \lambda_{i} a_{i}: \lambda_{i} \in \mathbb{N} \cup \left\lbrace 0\right\rbrace \ \text{for} \ i=  1, \ldots, k \ \text{with} \  \sum_{i=1}^{k} \lambda_{i}=h\right\rbrace,
	\]

	\[h^{\wedge}A:=\left\lbrace \sum_{i = 1}^{k} \lambda_{i} a_{i}: \lambda_{i} \in \left\lbrace 0, 1\right\rbrace \ \text{for} \ i=  1, \ldots, k \ \text{with} \  \sum_{i=1}^{k} \lambda_{i}=h\right\rbrace.\]
	
	Mistri and Pandey \cite{MISTRIPANDEY2014} generalized   $hA$ and  $h^{\wedge}A$, into  the   generalized $h$-fold  sumset,  denoted by $h^{\left( r\right) }A$, as follows:
	
	 Let $r$ be a positive integer such that $1 \leq r \leq h$. The  {\it generalized h-fold  sumset} $h^{\left( r\right) }A$, is defined by  \[h^{(r)}A:=\left\lbrace \sum_{i=1}^{k} \lambda_{i} a_{i} : 0 \leq \lambda_{i} \leq r \ \text{for} \ i=1, \ldots, k \ \text{with} \sum_{i=1}^{k} \lambda_{i}=h\right\rbrace.\]
	So, the generalized $h$-fold  sumset $h^{(r)}A$ is the set of all sums of  $h$ elements of $A$, in which every summand can repeat at most $r$ times. Therefore, $hA$ and $h^{\wedge}A$ are particular cases of $h^{( r)}A$ for $r=h$ and $r=1$, respectively.
	
	For a finite set  $H$ of positive integers, Bajnok \cite{BAJNOK2018}  introduced the sumset \[HA := \bigcup_{h \in H} hA,\] and the restricted sumset \[H^{\wedge}A := \bigcup_{h \in H} h^{\wedge}A.\]
	In a recent article, Bhanja and Pandey \cite{JBHANJA2021} considered a generalization of $HA$ and $H^{\wedge}A$, the {\it generalized H-fold sumset}, denoted by $H^{(r)}A$, defined by \[H^{(r)}A:=\bigcup_{h \in H}h^{(r)}A.\] Observed that, if $r\geq\max(H)$, then $H^{(r)}A=HA$ and if $r=1$, then $H^{(r)}A=H^{\wedge}A$. The sumset $H^{(r)}A$ becomes more  important as it also generalizes \textit{subset sums} and \textit{subsequence sums}.
	\noindent\subsection{\textit{Subset sum and Subsequence sum}} Let $A$ be a finite set of integers. The sum of all the elements of a given  subset $B$ of $A$ is called  \textit{subset sum} and it is denoted by $s(B)$. That is, \[s(B)=\sum_{b \in B}b.\] The set of all nonempty subset sum of $A$, denoted by $\sum(A)$, that is \[\sum(A)=\Big\{s(B): \emptyset \not= B\subseteq A\Big\}. \] Also we define, for $1 \leq \alpha \leq k$ \[\sum_{\alpha}(A)=\Big\{s(B): \emptyset \not= B\subseteq A \ \text{and} \ \left|B\right| \geq \alpha\Big\}.\]
	Similarly, we  define subsequence sum of a given sequence of integers.
	Let $A=\{a_{1}, a_{2}, \ldots,a_{k}\}$ be a set of $k$ integers and $r$ be a positive integer, with $a_{1} < a_{2} < \cdots < a_{k}$. Then we define a   sequence associated with $A$ as \[\mathbb{A}=(\underbrace{a_{1},\ldots,a_{1}}_{r-times},\underbrace{a_{2},\ldots,a_{2}}_{r-times},\ldots,\underbrace{a_{k},\ldots,a_{k}}_{r-times})=(a_{1},a_{2},\ldots,a_{k})_{r} (\text{say}).\]
	Let $\mathbb{B}$ be a subsequence of $\mathbb{A}$. Then \[\mathbb{B}=(\underbrace{a_{1},\ldots,a_{1}}_{r_{1}-times},\underbrace{a_{2},\ldots,a_{2}}_{r_{2}-times},\ldots,\underbrace{a_{k},\ldots,a_{k}}_{r_{k}-times}) \ \text{with} \ 0\leq r_{i}\leq r.\]
	Given any  subsequence $\mathbb{B}$ of $\mathbb{A}$, the sum of all terms of the subsequence $\mathbb{B}$ is called the \textit{subsequence sum}, is denoted by $s(\mathbb{B})$ and we write \[s(\mathbb{B})=\sum_{b \in \mathbb{B}}b.\]
	The set of all subsequence sums of a given sequence $\mathbb{A}$ is the set
	\[ \sum(\mathbb{A}) =  \left\lbrace s(\mathbb{B}): \text{ $\mathbb{B}$ is subsequence of $\mathbb{A}$ of length $\geq 1$}\right\rbrace.  \]
	For $1\leq \alpha\leq kr$,  define
	\[ \sum_{\alpha}(\mathbb{A}) =  \left\lbrace s(\mathbb{B}) : \text{ $\mathbb{B}$ is subsequence of $\mathbb{A}$ of length $\geq $} \ \alpha\right\rbrace.  \]
	
	Note that, we can write \[h^{(r)}A=\left\lbrace s(\mathbb{B}): \text{$\mathbb{B}$ is subsequence $\mathbb{A}$ of length $h$}\right\rbrace.\]
	With suitable sets $H$, we can express $\sum(A), \sum\limits_{\alpha}(A),\sum(\mathbb{A})$ and $\sum\limits_{\alpha}(\mathbb{A})$ in terms of $H^{\wedge}A$ and $H^{(r)}A$, as follows:\\
	$\clubsuit$ If $H=\{1,2,\ldots,k\}$, then $H^{\wedge}A = \bigcup_{h=1}^{k} h^{\wedge}A=\sum(A)$. \\
	$\clubsuit$ If $H= \{\alpha,\alpha+1,\ldots,k\}$,  then $H^{\wedge}A = \bigcup_{h=\alpha}^{k} h^{\wedge}A=\sum\limits_{\alpha}(A)$. \\	
	$\clubsuit$ If $H=\{1,2,\ldots,kr\}$, then $H^{(r)}A=\bigcup_{h=1}^{kr} h^{(r)}A=\sum(\mathbb{A})$.\\
	$\clubsuit$ If $H=\{\alpha,\alpha+1,\ldots,kr\}$, then $H^{(r)}A=\bigcup_{h=\alpha}^{kr} h^{(r)}A=\sum\limits_{\alpha}(\mathbb{A})$.
	
	Let $A=\{a_{1},a_{2},\ldots,a_{k}\}$ be a nonempty set of integers with $a_{1}<a_{2}<\cdots<a_{k}$. For an integer $c$, we write  $c \ast A = \{ ca :  a \in A \}$ and for integers $ a$ and   $b$ with $a < b$, we write $[a,b] = \{ a, a+1, \ldots, b \}$. For a nonempty set $S = \{ s_{1}, s_{2} , \ldots, s_{n-1}, s_{n} \} $, we let $\max(S), \min(S), \max_{-}(S), \min_{+}(S)$ be the largest,  smallest,  second largest  and   second smallest elements of $S$, respectively.  For a given real number $x$, $\lfloor x \rfloor$ and $\lceil x \rceil$ denote, floor function and ceiling function of $x$, respectively. We assume $\sum_{i=1}^{t} f(i) = 0$ if $t < 1$.
	
	Two standard problems associated with a sumset in additive number theory are to find best possible lower bound for  the cardinality of sumset when the set $A$ is known (called the direct problem) and to find the structure of the underlying set $A$ when the size of the sumset attains its lower bound (called the inverse problem). These two types of problems have been solved for the sumsets in  various types of groups.  We have several classical results on sumsets for  the case when $A$ is a subset of group of integers, (see \cite{BAJNOK2018},\cite{JAGBHANJA2021},\cite{MISTRIPANDEY2014}, \cite{MONO2015}, \cite{NATHAN1995},\cite{NATHAN1996},  \cite{YCHEN2015}),  and for subsequence sums and subset one may refer to \cite{BHANJA2021}, \cite{BHANJA2020}, \cite{JBHANJA2021},\cite{MISTRIPANDEYPRAKASH2015}. We mention now, some of these results that are applied in this paper.
	
	\begin{theorem}\textup{\cite[Theorem 1.3, Theorem 1.6]{NATHAN1996}}\label{Nathanson Theorem I}
		Let $h\geq 1$, and let $A$ be a nonempty finite set of integers. Then
		\begin{center}
			$\left| hA\right| \geq h\left| A\right|  - h + 1$.
		\end{center}
		This lower bound is best possible. Furthermore, if   $\left|hA\right|$ attains this lower bound with $h \geq 2 $, then $A$ is an arithmetic progression.
	\end{theorem}
	
	\begin{theorem} \textup{\cite[Theorem 1, Theorem  2]{NATHAN1995, NATHAN1996}}\label{Nathanson Theorem II}
		Let A be a nonempty finite set of integers, and let $1 \leq h \leq \left| A\right| $. Then
		\begin{center}
			$\left| h^{\wedge}A \right|  \geq h \left| A\right|  - h^{2} + 1$.
		\end{center}
		This lower bound is best possible. Furthermore, if   $\left|h^{\wedge}A\right|$ attains this lower bound with $\left|A \right| \geq 5$ and  $2 \leq h \leq |A| - 2$, then $A$ is an arithmetic progression.
	\end{theorem}
	Mistri and Pandey \cite{MISTRIPANDEY2014}  generalized  above results as follows:
	\begin{theorem} \textup{\cite[Theorem 2.1]{MISTRIPANDEY2014}}\label{MistriPandey Theorem I}
		Let A be a nonempty finite set of k integers. Let r and h be  integers such that $1 \leq r \leq h \leq kr$. Set $m= \left \lfloor{h/r}\right \rfloor$. Then
		\begin{center}
			$\left|h^{\left( r\right)}A \right| \geq mr\left( k-m\right)  + \left( h-mr\right) \left( k-2m-1\right)  +1$.
		\end{center}
		This lower bound is best possible.
	\end{theorem}
	
	\begin{theorem}\textup{\cite[Theorem 3.1, Theorem 3.2]{MISTRIPANDEY2014}}\label{MistriPandey Theorem II}
		Let $k \geq 3$. Let $r$ and $h \geq 2$ be  integers such
		that $1 \leq r \leq h \leq kr -2$ and $\left( k, h, r\right)  \neq \left( 4, 2, 1\right)$. Set $m= \left \lfloor{h/r}\right \rfloor$. If A is a finite set of k integers such that
		\begin{center}
			$\left|h^{\left( r\right)}A \right| = mr\left( k-m\right)  + \left( h-mr\right) \left( k-2m-1\right)  +1$,
		\end{center}
		then $A$ is an arithmetic progression.
	\end{theorem}
	Further generalization  of $h^{(r)}A$ was considered in \cite{MISTRIPANDEY2014} for which the direct and inverse results were proved by  Yang and Chen \cite{YCHEN2015}. Direct results for  $h^{(r)}A$ when $A$ is a subset of the group of residual classes modulo a prime and $A$ is a subset of a finite cyclic group were given, respectively, by Monopoli \cite{MONO2015} and Bhanja \cite{JAGBHANJA2021}.
	
	The direct and inverse theorems for the sumsets $HA$ and $H^{\wedge}A$  proved by  Bhanja \cite{BHANJA2021} are the following:
	
	\begin{theorem}\textup{\cite[Theorem 3]{BHANJA2021}}\label{Bhanja I}
		Let A be a set of k positive integers. Let H be a set of t positive integers  with $\max(H) = h_{t}$. Then
		\[\left|HA\right| \geq h_{t}(k-1) + t.\]
		This lower bound is optimal.
	\end{theorem}
	
	\begin{theorem}\textup{\cite[Theorem 5]{BHANJA2021}}\label{Bhanja II}
		Let A be a set of $k  \geq 2 $ positive integers and H be a set of $t  \geq 2 $ positive integers with $\max(H) = h_{t}$. If $$\left|HA\right| = h_{t}(k - 1)+ t,$$ then H is an arithmetic progression with common difference d and A is an arithmetic progression with common difference $d \ast  \min(A).$
	\end{theorem}
	\begin{theorem}\textup{\cite[Theorem 6, Corollary 7]{BHANJA2021}}\label{Bhanja III}
		Let A be a set of $k$ nonnegative integers and $H = \lbrace h_{1}, h_{2}, \ldots, h_{t} \rbrace $ be a set of	positive integers with $h_{1} < h_{2} < \cdots < h_{t}$. Set $h_{0}=0$. If $0\notin A$ and $h_{t} \leq k$, then
		\begin{equation*}
			\left|H^{\wedge}A\right|  \geq \sum_{i=1}^{t} (h_{i} - h_{i-1}) (k - h_{i}) + t.
		\end{equation*}
		If $0\in A$ and $  h_{t} \leq k-1$, then
		\begin{equation*}
			\left|H^{\wedge}A\right|  \geq h_{1}+ \sum_{i=1}^{t} (h_{i} - h_{i-1}) (k - h_{i}-1) + t.
		\end{equation*}
		The  lower bounds  are optimal.
	\end{theorem}
	\begin{theorem}\textup{\cite[Theorem 9, Corollary 10]{BHANJA2021}}\label{Bhanja IV}
		Let A be a set of k nonnegative integers. Let $H = \lbrace h_{1}, h_{2}, \ldots, h_{t} \rbrace $ be a set of	positive integers with $h_{1} < h_{2} < \cdots < h_{t}$. Set $h_{0}=0$. If $0 \notin A$, $k\geq 6$, $h_{t} \leq k-1$, and  \[ \left|H^{\wedge}A\right|  = \sum_{i=1}^{t} (h_{i} - h_{i-1}) (k - h_{i}) + t,\] then $H = h_{1} + [0,t-1]$ and $ A = \min(A) \ast [1,k]$.\\
		If $0 \in A$, $k\geq 7$, $h_{t} \leq k-2$, and   \[ \left|H^{\wedge}A\right|  = h_{1}+ \sum_{i=1}^{t} (h_{i} - h_{i-1}) (k - h_{i}) + t,\] then $H = h_{1} + [0,t-1]$ and $ A = \min(A\setminus\{0\}) \ast [0,k-1]$.
	\end{theorem}	
	In this paper, we prove similar direct and inverse results for the sumset $H^{(r)}A$ when $A$ is a finite nonempty set of positive integers. In sections 2 and 3, we prove our main theorems, Theorem \ref{Direct Theorem} and Theorem \ref{Inverse Theorem}, the direct  and inverse theorems for sumset $H^{(r)}A$, when $A$ is a finite set of positive integers.  Consequentaly we prove direct and inverse theorems when  $A$ contains nonnegative  integers with $0 \in A$.
	
	To state our main results we need some notation that are used throughout the paper. Let $H = \lbrace h_{1}, h_{2}, \ldots, h_{t} \rbrace$ be a set of positive integers with $0 = h_{0}< h_{1} < h_{2} < \cdots < h_{t}$ and $r$ be a positive integer.  If $t=1$, then $H^{(r)}A = h_{1}^{(r)}A$. So, we assume $t \geq 2$.  If $r>h_{t}$,  then $h^{(r)}_{i}A=h_{i}A$ for $1 \leq i \leq t$, giving  $H^{(r)}A =HA$. So we assume that $r \leq h_{t}$.  There always exists a unique positive integer  $l$ such that $h_{l-1} < r \leq h_{l}, \text{ where } 1 \leq l \leq t$. For $i=1,2, \ldots, t$, let $  h_{i}=m_{i}r+\epsilon_{i}, \text{ where } 0 \leq \epsilon_{i} \leq r-1$. For given set $H$ of positive integers and set of integers $A$ with $|H|=t$ and $|A|=k$, let
	\begin{multline}
		\mathcal{L}(H^{(r)}A) =    h_{l-1}(k-1)+(l-1)  + \sum_{i=l}^{t} r(m_{i}-m_{i-1})(k-m_{i})  \\ + \sum_{i=l}^{t}\Big((\epsilon_{i}-\epsilon_{i-1})(k-m_{i}-1)-\max\{\epsilon_{i},\epsilon_{i-1}\}(m_{i}-m_{i-1})+1 \Big).
	\end{multline}
	Note that, if  $0\leq i \leq l-1$, then $m_{i}=0$ and $\epsilon_{i}=h_{i}$. So,  we can also write \[	\mathcal{L}(H^{(r)}A) =    \sum_{i=1}^{t} \Big( r(m_{i}-m_{i-1})(k-m_{i})+(\epsilon_{i}-\epsilon_{i-1})(k-m_{i}-1)-\max\{\epsilon_{i},\epsilon_{i-1}\}(m_{i}-m_{i-1})+1 \Big).\]
	For $i=1, \ldots,t$, define
	\[M_{i}=\Big\lfloor \dfrac{h_{i}-h_{i-1}}{r} \Big \rfloor\]
	and for $j=0, \ldots,t-1$, define
	\[
	N_{j}=
	\begin{cases}
		\Big\lceil  \dfrac{h_{j}}{r} \Big\rceil  & \text{if } l-1 \leq j \leq t-1\\
		0 & \text{otherwise}.
	\end{cases}
	\]
	Also, let $\lbrace 0 \rbrace ^{(r)}A=\lbrace 0 \rbrace$.
	\section{Direct Theorems}
	\begin{theorem}\label{Direct Theorem}
		Let $A$ be a nonempty finite set of $k \geq 3$ positive integers.  Let $r$ be a positive integer and  $H$  be a  set  of  $t \geq 2$ positive integers with  $1 \leq r \leq \max(H) \leq  (k-1)r-1$. Then
		\begin{equation}\label{Direct Theorem  Eq- 1}
			\left|H^{(r)}A\right| \geq    \mathcal{L}(H^{(r)}A).
		\end{equation}
		This lower bound is best possible.
	\end{theorem}

	\begin{proof}
		Let $A=\lbrace a_{1}, a_{2}, \ldots, a_{k} \rbrace$ and $H=\lbrace h_{1},h_{2}, \ldots, h_{t}\rbrace$ be such that
		\[0<a_{1}<a_{2}< \cdots< a_{k} \ \text{and} \ 0=h_{0}<h_{1}< h_{2}<\cdots<h_{t}.\]
		For $i=0,1, \ldots, t $, write  $h_{i}=m_{i}r+ \epsilon_{i}, \ \text{where} \  0\leq \epsilon_{i} \leq r-1$. Then, we have $$0 = m_{0}\leq m_{1} \leq m_{2} \leq \cdots \leq m_{t} \leq  k-2.$$
		 Since  $l$ is a positive integer satisfying $h_{l-1} < r \leq h_{l}$, we have  $m_{i}=0$ and $\epsilon_{i}=h_{i}$ for $i=0, \ldots, l-1$. Set $S_{0}=\emptyset$.
		Define
		\begin{equation*}\label{generalized sumset eq 2}
			S_{i}=(h_{i}-h_{i-1})^{(r)}A_{i}+\max\lbrace h^{(r)}_{i-1}A \rbrace   \text{ for } i=1, 2, \ldots, t,
		\end{equation*}
		where 	
		\[A_{i}= \lbrace a_{1}, \ldots, a_{k-N_{i-1}} \rbrace \text{ for } i=1 , 2, \ldots, t.\]
		Note that $S_{i} \subseteq h^{(r)}_{i}A \subseteq H^{(r)}A$ and $\max(S_{i})<\min(S_{i+1})$ for all $i\in[1,t-1]$.  We shall define   sets $T_{i} \subseteq h_{i}^{(r)}A$ that satistfy $T_{i}\cap S_{i}=\emptyset$ for $i\in[0,t]$. Let $R_{i} = S_{i}\cup T_{i} \subseteq h_{i}^{(r)}A, \  \text{for} \ i= 0,1,\ldots,t$. If $i\in[0,l-1]$, then  define
		$T_{i}=\emptyset$. So, $|R_{0}| = 0$ for $l\geq 1$, and  by Theorem \ref{Nathanson Theorem I}, we have  $\left|R_{i}\right|=\left|S_{i}\right| \geq (h_{i}-h_{i-1})(k-1)+1$ for $l \geq 2$ and $i\in[1,l-1]$. If $i\in [l,t]$, then we define $T_{i}$  for every possible values of $\epsilon_{i-1}$ and $\epsilon_{i}$,  and consequently find $|R_{i}|$.\\
		$\clubsuit$ Let $i\in [l,t]$ be such that $\epsilon_{i-1}=0$ and $\epsilon_{i} \geq 0$. Then $M_{i}=m_{i}-m_{i-1}$ and $N_{i-1} = m_{i-1}$. Let $T_{i}=\emptyset$ in this case. Then  by Theorem \ref{MistriPandey Theorem I}, we have
		\begin{align*}
			\left|R_{i}\right|   &=	\left|S_{i}\right|  +  \left|T_{i}\right|\\ &	\geq M_{i}r(k-N_{i-1}-M_{i})+(h_{i}-h_{i-1}-M_{i}r)(k-N_{i-1}-2M_{i}-1)+1 \\ & = r(m_{i}-m_{i-1})(k-m_{i})+(\epsilon_{i}-\epsilon_{i-1})(k-m_{i}-1)-\epsilon_{i}(m_{i}-m_{i-1})+1.
		\end{align*}
		$\clubsuit$ Let $i\in [l,t]$ be such that $\epsilon_{i} > \epsilon_{i-1} >0$ and $m_{i}= m_{i-1}$. Then $M_{i}=m_{i}-m_{i-1}=0$ and $N_{i-1} = m_{i-1}+1$.
		For $j=0,1, \ldots, \epsilon_{i}-\epsilon_{i-1}$,  define
		\[T^{0}_{i,j}=(\epsilon_{i}-\epsilon_{i-1}-j)a_{k-m_{i}-1} + (\epsilon_{i-1}+j)a_{k-m_{i}} +  \sum_{p=1}^{m_{i}}ra_{k-m_{i}+p}.\]
		Then, we have  $\max(S_{i})=T^{0}_{i,0}< T^{0}_{i,1}< \cdots<T^{0}_{i,\epsilon_{i}-\epsilon_{i-1}}=\max(h_{i}^{(r)}A)<\min(S_{i+1})$. Let
		\begin{equation}\label{DT-eq-1}
			T_{i} = \Big\{T^{0}_{i,j} : j=1, \ldots, \epsilon_{i}-\epsilon_{i-1}\Big\}.	
		\end{equation}
		Then by Theorem \ref{MistriPandey Theorem I} and (\ref{DT-eq-1}), we have
		\begin{align*}
			\left|R_{i}\right| &=	\left|S_{i}\right|  +  \left|T_{i}\right|\\ &	\geq M_{i}r(k-N_{i-1}-M_{i})+(h_{i}-h_{i-1}-M_{i}r)(k-N_{i-1}-2M_{i}-1)+1 + (\epsilon_{i}-\epsilon_{i-1})\\ & = r(m_{i}-m_{i-1})(k-m_{i})+(\epsilon_{i}-\epsilon_{i-1})(k-m_{i}-1)-\epsilon_{i}(m_{i}-m_{i-1})+1.
		\end{align*}
		$\clubsuit$ Let $i\in [l,t]$ be such that $\epsilon_{i} > \epsilon_{i-1} > 0$ and $m_{i}= m_{i-1} + 1$. Then $M_{i}=m_{i}-m_{i-1}=1$ and $N_{i-1}=m_{i-1}+1=m_{i}$. For $j=0, \ldots, \epsilon_{i}-\epsilon_{i-1}-1$, define
		\[T^{0}_{i,j}=(\epsilon_{i}-\epsilon_{i-1}-j)a_{k-m_{i-1}-2}+ ra_{k-m_{i-1}-1} + (\epsilon_{i-1}+j)a_{k-m_{i-1}} +  \sum_{p=1}^{m_{i-1}}ra_{k-m_{i-1}+p},\]
		\[T^{1}_{i,j}=(\epsilon_{i}-\epsilon_{i-1}-j)a_{k-m_{i-1}-2}+ (r-1)a_{k-m_{i-1}-1} + (\epsilon_{i-1}+j+1)a_{k-m_{i-1}} +  \sum_{p=1}^{m_{i-1}}ra_{k-m_{i-1}+p} \]
		and for $j=0, 1, \ldots, r-\epsilon_{i}$,
		\[U^{0}_{i,j}= (r-j)a_{k-m_{i-1}-1} + (\epsilon_{i}+j)a_{k-m_{i-1}} +  \sum_{p=1}^{m_{i-1}}ra_{k-m_{i-1}+p}.\]
		Then, we have  $\max(S_{i})=T^{0}_{i,0}<T^{1}_{i,0}<T^{0}_{i,1}<T^{1}_{i,1}< \cdots < T^{0}_{i,\epsilon_{i}-\epsilon_{i-1}-1} < T^{1}_{i,\epsilon_{i}-\epsilon_{i-1}-1} < U^{0}_{i,0}<U^{0}_{i,1} < \cdots < U^{0}_{i,r-\epsilon_{i}}=\max(h_{i}^{(r)}A) <\min(S_{i+1})$. Assume $\Big\{T^{0}_{i,j}: j=1,\ldots,\epsilon_{i}-\epsilon_{i-1}-1 \Big\} = \emptyset,$ if $\epsilon_{i}-\epsilon_{i-1}=1$.  Let
		\begin{equation}\label{DT-eq-2}
			T_{i}= \Big\{T^{0}_{i,j}: j=1,\ldots, \epsilon_{i}-\epsilon_{i-1}-1 \Big\}  \bigcup \Big\{T^{1}_{i,j}: j=0,\ldots, \epsilon_{i}-\epsilon_{i-1}-1 \Big\}  \bigcup \Big\{U^{0}_{i,j}: j=0,\ldots, r-\epsilon_{i} \Big\}.	
		\end{equation}
		Then by Theorem \ref{MistriPandey Theorem I} and (\ref{DT-eq-2}), we have
		\begin{align*}
			\left|R_{i}\right| &=	\left|S_{i}\right|  +  \left|T_{i}\right|\\ & \geq M_{i}r(k-N_{i-1}-M_{i})+(h_{i}-h_{i-1}-M_{i}r)(k-N_{i-1}-2M_{i}-1)+1 + 2(\epsilon_{i}-\epsilon_{i-1})+(r-\epsilon_{i})\\
			& = r(m_{i}-m_{i-1})(k-m_{i})+(\epsilon_{i}-\epsilon_{i-1})(k-m_{i}-1)-\epsilon_{i}(m_{i}-m_{i-1})+1.
		\end{align*}
		$\clubsuit$ Let $i\in [l,t]$ be such that $\epsilon_{i} > \epsilon_{i-1} > 0$ and $m_{i}> m_{i-1} + 1$. Then $M_{i}=m_{i}-m_{i-1} \geq 2$ and $N_{i-1}=m_{i-1}+1$. For $j=0, \ldots,\epsilon_{i}-\epsilon_{i-1}-1$ and $q=1,  \ldots, m_{i}-m_{i-1}$, define \[T^{0}_{i,j}=(\epsilon_{i}-\epsilon_{i-1}-j)a_{k-m_{i}-1} + \left( \sum_{p=1}^{m_{i}-m_{i-1}}ra_{k-m_{i}-1+p} \right) + (\epsilon_{i-1}+j)a_{k-m_{i-1}} +  \sum_{p=1}^{m_{i-1}}ra_{k-m_{i-1}+p},\]
		\begin{multline*}
			T^{q}_{i,j} =  (\epsilon_{i}-\epsilon_{i-1}-j)a_{k-m_{i}-1} + \left( \sum_{p=1, \  p \neq m_{i}-m_{i-1}+1-q }^{m_{i}-m_{i-1}}ra_{k-m_{i}-1+p} \right) + (r-1)a_{k-m_{i-1}-q} \\ + (\epsilon_{i-1}+j+1)a_{k-m_{i-1}} +  \sum_{p=1}^{m_{i-1}}ra_{k-m_{i-1}+p}.	
		\end{multline*}
		For $j=0,\ldots,r-\epsilon_{i}-1$ and $q=1, \ldots, m_{i}-m_{i-1}-1$,  define
		\[U^{0}_{i,j}=(r-j)a_{k-m_{i}} + \sum_{p=1}^{m_{i}-m_{i-1}-1}ra_{k-m_{i}+p} + (\epsilon_{i}+j)a_{k-m_{i-1}} + \sum_{p=1}^{m_{i-1}}ra_{k-m_{i-1}+p},\]
		\begin{multline*}
			U^{q}_{i,j}=(r-j)a_{k-m_{i}}+ \left(\sum_{p=1, \ p \neq m_{i}-m_{i-1}-q}^{m_{i}-m_{i-1}-1} ra_{k-m_{i}+p}\right)  + (r-1)a_{k-m_{i-1}-q}+(\epsilon_{i}+j+1)a_{k-m_{i-1}}\\ +  \sum_{p=1}^{m_{i-1}}ra_{k-m_{i-1}+p}.
		\end{multline*}
		Furthermore, define   \[U^{0}_{i,r-\epsilon_{i}}=\epsilon_{i}a_{k-m_{i}}+ \sum_{p  =  k-m_{i}+1}^{k} ra_{p}.\]
		Then  $U^{0}_{i,r-\epsilon_{i}} = \max(h_{i}^{(r)}A) <\min(S_{i+1})$ and
		\begin{equation*}
			\begin{array}{cccccccccccc}
				\max(S_{i})=T^{0}_{i,0} & < &T^{1}_{i,0} & < &\cdots  & < &T^{m_{i}-m_{i-1}}_{i,0} & < & \\ [5pt] T^{0}_{i,1} & < &T^{1}_{i,1} & < &\cdots  & < &T^{m_{i}-m_{i-1}}_{i,1} &<&\\ [7pt]
				\vdots&\vdots&\vdots&\vdots&  & \vdots & \vdots & \vdots\\ [7pt]
				T^{0}_{i,\epsilon_{i}-\epsilon_{i-1}-1} & <& T^{1}_{i,\epsilon_{i}-\epsilon_{i-1}-1}&<&\cdots&<&T^{m_{i}-m_{i-1}}_{i,\epsilon_{i}-\epsilon_{i-1}-1}&< &\\ [7pt]
				U^{0}_{i,0}&<&U^{1}_{i,0}&<& \cdots&<&U^{m_{i}-m_{i-1}-1}_{i,0} &<&\\ [7pt]
				U^{0}_{i,1}&<&U^{1}_{i, 1}&<& \cdots &<& U^{m_{i}-m_{i-1}-1}_{i, 1} &<&\\ [7pt]
				\vdots&\vdots&\vdots&\vdots&  & \vdots & \vdots & \vdots \\ [7pt]
				U^{0}_{i, r-\epsilon_{i}-1}&<&U^{1}_{i, r-\epsilon_{i}-1}&<&\cdots&<&U^{m_{i}-m_{i-1}-1}_{i, r-\epsilon_{i}-1} & < &U^{0}_{i,r-\epsilon_{i}}.
			\end{array}
		\end{equation*}
		Assume $\Big\{T^{0}_{i,j}: j=1,\ldots,\epsilon_{i}-\epsilon_{i-1}-1 \Big\} = \emptyset,$ if $\epsilon_{i}-\epsilon_{i-1}=1$.
		Let
		\begin{align}\label{DT-eq-3}
			T_{i} =&\Big\{T^{0}_{i,j}: j=1,\ldots,\epsilon_{i}-\epsilon_{i-1}-1 \Big\}    \bigcup \Big\{T^{q}_{i,j}: j=0,\ldots,\epsilon_{i}-\epsilon_{i-1}-1 \ \text{and} \ q=1,\ldots,m_{i}-m_{i-1}\Big\} \nonumber\\ & \bigcup \Big\{U^{0}_{i,j}: j=0,\ldots,r-\epsilon_{i} \Big\}     \bigcup \Big\{U^{q}_{i,j}: j=0,\ldots,r-\epsilon_{i}-1 \ \text{and} \ q=1,\ldots,m_{i}-m_{i-1}-1\Big\}. 		
		\end{align}
		Then by Theorem \ref{MistriPandey Theorem I} and (\ref{DT-eq-3}), we have
		\begin{align*}
			\left|R_{i}\right| &=	\left|S_{i}\right|  +  \left|T_{i}\right|\\ &\geq M_{i}r(k-N_{i-1}-M_{i})+(h_{i}-h_{i-1}-M_{i}r)(k-N_{i-1}-2M_{i}-1)+1\\ & \ \ \ \ \ \ \ \ \ \ \ \ \ \ \ \ \ \ \ \ \ \ \ \ +  (\epsilon_{i}-\epsilon_{i-1})(m_{i}-m_{i-1}+1)+(r-\epsilon_{i})(m_{i}-m_{i-1})\\ &  =r(m_{i}-m_{i-1})(k-m_{i})+(\epsilon_{i}-\epsilon_{i-1})(k-m_{i}-1)-\epsilon_{i}(m_{i}-m_{i-1})+1.
		\end{align*}
		$\clubsuit$ Let $i\in [l,t]$ be such that $\epsilon_{i} = \epsilon_{i-1} > 0$ and $m_{i}= m_{i-1} + 1$. Then $M_{i}=m_{i}-m_{i-1}=1$ and $N_{i-1}=m_{i-1}+1$.
		For $j=0, \ldots, r-\epsilon_{i}$, define
		\[U^{0}_{i,j}= (r-j)a_{k-m_{i-1}-1} + (\epsilon_{i}+j)a_{k-m_{i-1}} +  \sum_{p=1}^{m_{i-1}}ra_{k-m_{i-1}+p}.\]
		Then $\max(S_{i})= U^{0}_{i,0}<U^{0}_{i,1}  < \cdots < U^{0}_{i,r-\epsilon_{i}} = \max(h_{i}^{(r)}A) < \min(S_{i+1})$. Let
		\begin{equation}\label{DT-eq-4}
			T_{i}= \Big\{\ U^{0}_{i,j} : j=1,\ldots,r-\epsilon_{i}\Big\}.	
		\end{equation}
		Then by Theorem \ref{MistriPandey Theorem I} and (\ref{DT-eq-4}), we have
		\begin{align*}
			\left|R_{i}\right| &=	\left|S_{i}\right|  +  \left|T_{i}\right|\\ & \geq M_{i}r(k-N_{i-1}-M_{i})+(h_{i}-h_{i-1}-M_{i}r)(k-N_{i-1}-2M_{i}-1)+1 +(r-\epsilon_{i})\\
			& = r(m_{i}-m_{i-1})(k-m_{i})+(\epsilon_{i}-\epsilon_{i-1})(k-m_{i}-1)-\epsilon_{i}(m_{i}-m_{i-1})+1.
		\end{align*}
		$\clubsuit$ Let $i\in [l,t]$ be such that $\epsilon_{i} = \epsilon_{i-1} > 0$ and $m_{i}> m_{i-1} + 1$. Then $M_{i}=m_{i}-m_{i-1} \geq 2$ and $N_{i-1}=m_{i-1}+1$.
		For $j=0,\ldots,r-\epsilon_{i}-1$ and $q=1, \ldots, m_{i}-m_{i-1}-1$, define
		\[U^{0}_{i,j}=(r-j)a_{k-m_{i}}+\left(\sum_{p=1}^{m_{i}-m_{i-1}-1} ra_{k-m_{i}+p}\right)+(\epsilon_{i}+j)a_{k-m_{i-1}}+ \sum_{p=1}^{m_{i-1}}ra_{k-m_{i-1}+p}\]
		and	
		\begin{multline*}
			U^{q}_{i,j}=(r-j)a_{k-m_{i}}+ \left(\sum_{p=1, \ p \neq m_{i}-m_{i-1}-q}^{m_{i}-m_{i-1}-1} ra_{k-m_{i}+p}\right) + (r-1)a_{k-m_{i-1}-q}+(\epsilon_{i}+j+1)a_{k-m_{i-1}} \\
			+  \sum_{p=1}^{m_{i-1}}ra_{k-m_{i-1}+p}.
		\end{multline*}
		Furthermore, define  \[U^{0}_{i,r-\epsilon_{i}}=\epsilon_{i}a_{k-m_{i}}+ \sum_{p=1}^{m_{i}} ra_{k-m_{i}+p}.\]
		It is easy to see that  $U^{0}_{i,r-\epsilon_{i}} = \max(h_{i}^{(r)}A) <\min(S_{i+1})$ and
		\begin{equation*}
			\begin{array}{cccccccccc}
				\max(S_{i})=U^{0}_{i,0} & < & U^{1}_{i,0} & < & \cdots & < & U^{m_{i}-m_{i-1}-1}_{i,0} & <  \\ [7pt]
				U^{0}_{i,1} & < & U^{1}_{i, 1} & < & \cdots & < & U^{m_{i}-m_{i-1}-1}_{i, 1} & <  \\ [7pt]
				\vdots &\vdots&\vdots&\vdots&&\vdots&\vdots&\vdots & \\ [7pt]
				U^{0}_{i, r-\epsilon_{i}-1} & < & U^{1}_{i, r-\epsilon_{i}-1} & < & \cdots &< & U^{m_{i}-m_{i-1}-1}_{i, r-\epsilon_{i}-1} & < & U^{0}_{i,r-\epsilon_{i}}.
			\end{array}
		\end{equation*}
		Let
		\begin{equation}\label{DT-eq-5}
			T_{i} = \left\{ U^{0}_{i,j} :  j=1,\ldots, r-\epsilon_{i} \right\} \bigcup \Big \{ U^{q}_{i,j} :    j=0, \ldots,r-\epsilon_{i}-1 \ \text{and} \ q=1,\ldots,m_{i}-m_{i-1}-1  \Big\}.
		\end{equation}
		Then by Theorem \ref{MistriPandey Theorem I} and (\ref{DT-eq-5}), we have
		\begin{align*}
			\left|R_{i}\right| &=	\left|S_{i}\right|  +  \left|T_{i}\right|\\ & \geq M_{i}r(k-N_{i-1}-M_{i})+(h_{i}-h_{i-1}-M_{i}r)(k-N_{i-1}-2M_{i}-1)+1 +(r-\epsilon_{i})(m_{i}-m_{i-1})\\
			& = r(m_{i}-m_{i-1})(k-m_{i})+(\epsilon_{i}-\epsilon_{i-1})(k-m_{i}-1)-\epsilon_{i}(m_{i}-m_{i-1})+1.
		\end{align*}
		$\clubsuit$ Let $i\in [l,t]$ be such that $ \epsilon_{i} < \epsilon_{i-1} $ and $m_{i}= m_{i-1} + 1$. Then $M_{i}=m_{i}-m_{i-1}-1 = 0$ and $N_{i-1}=m_{i-1}+1=m_{i}$. For $j=0, \ldots, r-\epsilon_{i-1}$, define \[T^{0}_{i,j}=(r+\epsilon_{i}-\epsilon_{i-1}-j)a_{k-m_{i-1}-1} + (\epsilon_{i-1}+j)a_{k-m_{i-1}} +  \sum_{p=1}^{m_{i-1}}ra_{k-m_{i-1}+p}.\]	
		Then $\max(S_{i}) = T^{0}_{i,0}<T^{0}_{i,1}<\cdots<T^{0}_{i,r-\epsilon_{i-1}} = \max(h_{i}^{(r)}A) < \min(S_{i+1})$.
		Let
		\begin{equation}\label{DT-eq-6}
			T_{i}=\{T^{0}_{i,j} : j=1, \ldots, r-\epsilon_{i-1}\}.
		\end{equation}
		Then by Theorem \ref{MistriPandey Theorem I} and (\ref{DT-eq-6}), we have
		\begin{align*}
			\left|R_{i}\right| &=\left|S_{i}\right|  +  \left|T_{i}\right|\\ & \geq M_{i}r(k-N_{i-1}-M_{i})+(h_{i}-h_{i-1}-M_{i}r)(k-N_{i-1}-2M_{i}-1)+1 +(r-\epsilon_{i-1})\\
			& = r(m_{i}-m_{i-1})(k-m_{i})+(\epsilon_{i}-\epsilon_{i-1})(k-m_{i}-1)-\epsilon_{i-1}(m_{i}-m_{i-1})+1.
		\end{align*}
		$\clubsuit$ Let $i\in [l,t]$ be such that $ \epsilon_{i} < \epsilon_{i-1} $ and $m_{i}> m_{i-1} + 1$. Then $M_{i}=m_{i}-m_{i-1}-1 \geq 1$ and $N_{i-1}=m_{i-1}+1$.	For $j=0, \ldots,r-\epsilon_{i-1}-1$ and $q=1, \ldots, m_{i}-m_{i-1}-1$,  define
		\[T^{0}_{i,j}=(r+\epsilon_{i}-\epsilon_{i-1}-j)a_{k-m_{i}} + \left( \sum_{p=1}^{m_{i}-m_{i-1}-1}ra_{k-m_{i}+p} \right) + (\epsilon_{i-1}+j)a_{k-m_{i-1}} +  \sum_{p=1}^{m_{i-1}}ra_{k-m_{i-1}+p},\]
		\begin{multline*}
			T^{q}_{i,j}=(r+\epsilon_{i}-\epsilon_{i-1}-j)a_{k-m_{i}} + \left( \sum_{p=1, \ p \neq m_{i}-m_{i-1}-q }^{m_{i}-m_{i-1}-1}ra_{k-m_{i}+p} \right) + (r-1)a_{k-m_{i-1}-q} \\+ (\epsilon_{i-1}+j+1)a_{k-m_{i-1}} +  \sum_{p=1}^{m_{i-1}}ra_{k-m_{i-1}+p}.	
		\end{multline*}
		Define also
		\[T^{0}_{i,r-\epsilon_{i-1}}=\epsilon_{i}a_{k-m_{i}}+ \sum_{p=1}^{m_{i}} ra_{k-m_{i}+p}.\]
		It is easy to see that $T^{0}_{i,r-\epsilon_{i-1}}  = \max(h_{i}^{(r)}A) <\min(S_{i+1})$ and
		\begin{equation*}
			\begin{array}{ccccccccccc}
				\max(S_{i})=T^{0}_{i,0} & < & T^{1}_{i,0} & < & \cdots & < & T^{m_{i}-m_{i-1}-1}_{i,0} & < \\ [7pt]
				T^{0}_{i,1} & < & T^{1}_{i,1} &< &  \cdots  & < & T^{m_{i}-m_{i-1}-1}_{i,1} & < \\ [7pt]
				\vdots& \vdots & \vdots& \vdots & & \vdots & \vdots& \vdots \\ [7pt]
				T^{0}_{i, r-\epsilon_{i-1}-1} & < & T^{1}_{i, r-\epsilon_{i-1}-1} & < & \cdots & < & T^{m_{i}-m_{i-1}-1}_{i, r-\epsilon_{i-1}-1} &< &T^{0}_{i,r-\epsilon_{i-1}}  <\min(S_{i+1}).
			\end{array}
		\end{equation*}
		Let
		\begin{equation}\label{DT-eq-7}
			T_{i} =\Big\{T_{i,j}^{0}: j=1, \ldots, r-\epsilon_{i-1}\Big\} \bigcup \Big\{T_{i,j}^{q}:j=0,\ldots, r-\epsilon_{i-1}-1 \ \text{and} \ q=1, \ldots, m_{i}-m_{i-1}-1\Big\}.
		\end{equation}
		Then by Theorem \ref{MistriPandey Theorem I} and (\ref{DT-eq-7}), we have
		\begin{align*}
			\left|R_{i}\right| &=	\left|S_{i}\right|  +  \left|T_{i}\right|\\ & \geq M_{i}r(k-N_{i-1}-M_{i})+(h_{i}-h_{i-1}-M_{i}r)(k-N_{i-1}-2M_{i}-1)+1+(r-\epsilon_{i-1})(m_{i}-m_{i-1})\\ &  > r(m_{i}-m_{i-1})(k-m_{i})+(\epsilon_{i}-\epsilon_{i-1})(k-m_{i}-1)-\epsilon_{i-1}(m_{i}-m_{i-1})+1.
		\end{align*}	
		Hence
		\begin{align*}
			&\left|H^{(r)}A\right| \geq  \sum_{i=0}^{t}\left|R_{i}\right|  = \sum_{i=0}^{l-1}\left|S_{i}\right|+\sum_{i=l}^{t}\left|S_{i}\cup T_{i}\right| \\ & \geq \sum_{i=1}^{l-1} (h_{i}-h_{i-1})(k-1)+1 \nonumber  \\ & \ \ \ \ \ \ \ \ \ + \sum_{i=l}^{t} r(m_{i}-m_{i-1})(k-m_{i}) + (\epsilon_{i}-\epsilon_{i-1})(k-m_{i}-1)-\max\{\epsilon_{i}, \epsilon_{i-1}\}(m_{i}-m_{i-1}) + 1\nonumber\\	&= h_{l-1}(k-1)+(l-1) \nonumber  \\ & \ \ \ \ \ \ \ \ + \sum_{i=l}^{t} r(m_{i}-m_{i-1})(k-m_{i)}+(\epsilon_{i}-\epsilon_{i-1})(k-m_{i}-1)-\max\{\epsilon_{i}, \epsilon_{i-1}\}(m_{i}-m_{i-1}) + 1 \\
			& =	\mathcal{L}(H^{(r)}A).
		\end{align*}
		This proves (\ref{Direct Theorem  Eq- 1}). Next, we show that this bound  is best possible. Let  $H = [1,(k-1)r-1]$, $A=   \{ 1, 2, \ldots, k \}$. Then $H^{(r)}A \subseteq [1,2(r-1)+3r+ \cdots + kr]$. So $\left|H^{(r)}A\right| \leq \dfrac{rk(k+1)}{2} - r-2$. On the other hand, we have by (\ref{Direct Theorem  Eq- 1}),   $\left|H^{(r)}A\right| \geq \dfrac{rk(k+1)}{2} - r-2$. This completes the proof of Theorem \ref{Direct Theorem}.
		
	\end{proof}
	\begin{remark}\label{Remark 2.2}
		Following the notation from Theorem \ref{Direct Theorem}.
		\begin{enumerate}

			\item [\upshape(a)] If
			$0=h_{0}<h_{1}<\cdots<h_{t_{0}-1}<(k-1)r \leq h_{t_{0}}<\cdots<h_{t} \leq kr$ with $t_{0} \geq 2$, then we have
			\[\max(h^{(r)}_{t_{0}-1}A)<\max\nolimits_{-}(h^{(r)}_{t_{0}}A)<\max (h^{(r)}_{t_{0}}A)<\max(h^{(r)}_{t_{0}+1}A)<\cdots<\max(h^{(r)}_{t}A). \]
			So  $\left|H^{(r)}A\right| \geq \left|H^{(r)}_{t_{0}-1}A\right| +  t-t_{0}+2 \geq \mathcal{L}(H_{t_{0}-1}^{(r)}A)+t-t_{0}+2$, where $H_{t_{0}-1}=\{h_{1}, \ldots,h_{t_{0}-1}\},  t_{0} \geq 2 $.
			This  lower bound is best possible, as that  can be verified with  $ A=[1,k]$ and $H=[1,rk]$. Clearly, we have $\left|H^{(r)}A\right|=\dfrac{rk(k+1)}{2}$.
			
			\item [\upshape(b)] If $0=h_{0}<(k-1)r\leq h_{1}<\cdots<h_{t}\leq kr$, then
			\[H^{(r)}A \supseteq h_{1}^{(r)}A \cup \{\max(h_{i}^{(r)}A) : i= 2, \ldots,t\}.\]
			Therefore
			\[\left|H^{(r)}A\right| \geq | h_{1}^{(r)}A | + t-1 \geq  m_{1}r(k-m_{1})+(h_{1}-m_{1}r)(k-2m_{1}-1)+t,\] where $m_{1}=\lfloor h_{1}/r \rfloor$. To check, this bound is best possible, we take $A=[1,k]$ and $H=[(k-1)r,kr]$. Then $H^{(r)}A=[r+2r+\cdots+(k-1)r,r+2r+\cdots+kr]$ and hence $\left|H^{(r)}A\right|=kr+1$.
		\end{enumerate}	
	\end{remark}

	\begin{cor}\label{Corollary 2.1}
		Let   $A$ be a nonempty finite set of $k\geq 4$ nonnegative  integers with $0 \in A$. Let $r$ be a positive integer and  $H$  be a  set  of  $t \geq 2$ positive integers with  $1 \leq r \leq \max (H) \leq (k-2)r-1 $. Let  $m = \lceil \min(H)/r  \rceil$ and $m_{1}=\lfloor \min(H)/r \rfloor$. Then
		\begin{equation}\label{Eq- coro-2.1}
			\left|H^{(r)}A\right| \geq    m_{1}r(m-m_{1}+1) + (\min(H)-m_{1}r)(m-2m_{1}) + \mathcal{L}\Big(H^{(r)}(A\setminus\{0\}) \Big).
		\end{equation}
		This lower bound is best possible.
	\end{cor}
	
	\begin{proof}
		Let $A=\{0,a_{1},\ldots,a_{k-1}\}$ be a set of nonnegative integers with $0<a_{1}< \cdots <a_{k-1}$ and  $H = \lbrace h_{1}, h_{2}, \ldots, h_{t} \rbrace$ be a set of positive integers with $0 = h_{0}< h_{1}=\min(H) < h_{2} < \cdots < h_{t}=\max(H)$. Consider  $A^{'} = A \setminus \lbrace 0 \rbrace$.  Then $H^{(r)}A^{'} \subseteq H^{(r)}A$.
		
		Let $m = \lceil h_{1}/r \rceil$, $h_{1}=m_{1}r+\epsilon_{1}$, where $0 \leq \epsilon_{1} \leq r-1$  and $B=\{0,a_{1}, \ldots, a_{m}\}$. Then \[h^{(r)}_{1}B \subseteq H^{(r)}A\]
		and $h^{(r)}_{1}B \cap H^{(r)}A^{'} = \max(h^{(r)}_{1}B) = \min (H^{(r)}A^{\prime})= ra_{1} + \cdots + ra_{m_{1}} + \epsilon_{1}a_{m_{1}+1} $. Hence by Theorem \ref{MistriPandey Theorem I} and Theorem \ref{Direct Theorem}, we have
		\begin{align*}
			\left|H^{(r)}A\right| &\geq | h^{(r)}_{1}B |  + \left|H^{(r)}A^{\prime}\right| -1  \\ &\geq m_{1}r(m-m_{1}+1) + (\min(H)-m_{1}r)(m-2m_{1}) + \mathcal{L}\Big(H^{(r)}(A\setminus\{0\}) \Big) .
		\end{align*}
		This proves the Corollary. To check optimallity of the bound, take $A=[0,k-1]$ and $H=[1,(k-2)r-1] $. Then $H^{(r)}A \subseteq [0,2(r-1)+3r+ \cdots+(k-1)r]$ and $\left|H^{(r)}A\right| \leq \dfrac{rk(k-1)}{2}-r-1$. From (\ref{Eq- coro-2.1}), we have  $\left|H^{(r)}A\right| \geq \dfrac{rk(k-1)}{2}-r-1$.
	\end{proof}
	
	\begin{remark}\label{Remark 2.3}
		Following the notation from Corollary \ref{Corollary 2.1}.
		\begin{enumerate}
			\item [\upshape(a)]  If
			$0=h_{0}<h_{1}<\cdots<h_{t_{0}-1}<(k-2)r \leq h_{t_{0}}<\cdots<h_{t} \leq (k-1)r$ with $t_{0} \geq 2$, then we have
			\[\ \ \ \ \ \ \max(h^{(r)}_{t_{0}-1}A)<\max\nolimits_{-}(h^{(r)}_{t_{0}}A)<\max (h^{(r)}_{t_{0}}A)<\max(h^{(r)}_{t_{0}+1}A)<\cdots<\max(h^{(r)}_{t}A). \]
			So
			$\left|H^{(r)}A\right| \geq    m_{1}r(m-m_{1}+1) + (\min(H)-m_{1}r)(m-2m_{1}) + \mathcal{L}\Big(H^{(r)}(A\setminus\{0\}) \Big) +  t-t_{0}+2$.
			This  lower bound is best possible, as that  can be verified with  $ A=[0,k-1]$ and $H=[1,(k-1)r]$. Clearly, we have $\left|H^{(r)}A\right|=\dfrac{rk(k-1)}{2}+1$. Also, if we take $H = [1,(k-1)r] \cup X$, where $X \subseteq [(k-1)r + 1, kr]$, then again $\left|H^{(r)}A\right|=\dfrac{rk(k-1)}{2}+1$.
			
			\item [\upshape(b)] If $0=h_{0}<(k-2)r \leq h_{1}<\cdots<h_{t}\leq (k-1)r$, then
			\[H^{(r)}A \supseteq h_{1}^{(r)}A \cup \{\max(h_{i}^{(r)}A) : i= 2, \ldots,t\}.\] Therefore \[\left|H^{(r)}A\right| \geq | h_{1}^{(r)}A | + t-1 \geq m_{1}r(k-m_{1})+(h_{1}-m_{1}r)(k-2m_{1}-1)+t,\] where $m_{1}=\lfloor h_{1}/r \rfloor$. To check, this bound is best possible, we take $A=[0,k-1]$ and $H=[(k-2)r,(k-1)r]$. Then $H^{(r)}A=[r+2r+\cdots+(k-3)r,r+2r+\cdots+(k-1)r]$ and hence $\left|H^{(r)}A\right|=(2k-3)r+1$.
		\end{enumerate}	
	\end{remark}

	\begin{remark}
		\begin{enumerate}
			\item [\upshape(a)] For $r = \max(H) = h_{t}$, Theorem \ref{Bhanja I} follows from    Theorem \ref{Direct Theorem}  as a consequence.
			
			\item [\upshape(b)]  For  $r=1$,  Theorem \ref{Bhanja III} follows from  Remark \ref{Remark 2.2}  and Remark \ref{Remark 2.3}   as a consequence.
		\end{enumerate}
		
	\end{remark}

	\section{Inverse problem}
	This section deals with the inverse theorems associated with the sumset $H^{(r)}A$. In this section, we charaterize the  sets $A$ and $H$, when the cardinality of $H^{(r)}A$  is equal to its optimal lower bound. There are some cases in which either   $A$ or $H$ or both  may not be arithmetic progression  but size of $H^{(r)}A$ is equal to the  optimal lower bound (called extremal set). See some  extremal sets in  \cite[Section 3]{MISTRIPANDEY2014} and \cite[Section 2.2]{BHANJA2020}. Here we give some more  example of extremal sets.
	
	\begin{enumerate}
		
		\item Let $A$ be a set of $k\left(\geq 3\right)$ integers and $r$ be a positive integer. If $H=\{1,rk\}$ or $H= \{rk-1,rk\}$,  then $|H^{(r)}A|=k+1$.
		
		\item Let $A=\{a_{1},a_{2},a_{1}+a_{2}\}$ with  $0<a_{1}<a_{2}$ and $H\subseteq\{1,2,3\}$ with $r=1$; or $A=\{0,a_{1},a_{2},a_{1}+a_{2},\}$ with $H \subseteq \{1,2,3\}$ and $r=1$. Then the sets $A$ are extremal sets.
	\end{enumerate}
	We now present the main inverse results associated with $H^{(r)}A$.

	\begin{theorem}\label{Inverse Theorem}
		Let $r \geq 1$ be a positive integer, $A$ be a nonempty finite set of $k \geq 6$ positive integers and $H$  be a  set  of  $t \geq 2$ positive integers with $1\leq r \leq \max(H)  \leq (k-1)r-1  $.  If 	
		\begin{equation*}\label{Inverse Theorem -Eq-1}
			\left|H^{(r)}A\right| =   \mathcal{L}(H^{(r)}A),
		\end{equation*}
		then $H$ is  an arithmetic progression with common difference $d \leq r$ and  $A$ is an arithmetic progression with common difference $d\ast \min(A)$.
	\end{theorem}

	\begin{proof}
		
		Let $A=\lbrace a_{1}, a_{2}, \ldots, a_{k} \rbrace$ and $H=\lbrace h_{1},h_{2}, \ldots, h_{t}\rbrace$ be such that \[0<a_{1}<a_{2}< \cdots< a_{k} \text{ and } 0=h_{0}<h_{1}< h_{2}<\cdots<h_{t}.\]  For $i=1, \ldots, t$, let $  h_{i}=m_{i}r+\epsilon_{i}, \text{ where } 0 \leq \epsilon_{i} \leq r-1$.
		Let $l$ be a positive integer such  that $h_{l-1} < r \leq h_{l}, \text{ where } 1 \leq l \leq t$.
		Since  $|H^{(r)}A|$ is equal to its lower bound given in  (\ref{Direct Theorem  Eq- 1}), we have, from the proof of  Theorem \ref{Direct Theorem} that,   $|H^{(r)}A|=\sum_{i=1}^{t} \left|R_{i}\right|$. This implies that \[\left|R_{1}\right| = \left| h^{(r)}_{1}A \right|=m_{1}r(k-m_{1})+(h_{1}-m_{1}r)(k-2m_{1}-1)+1\]
		and
		 $\left|R_{i}\right| = \left|S_{i}\right|+\left|T_{i}\right|=  r(m_{i}-m_{i-1})(k-m_{i})+(\epsilon_{i}-\epsilon_{i-1})(k-m_{i}-1)-\max\{\epsilon_{i},\epsilon_{i-1}\}(m_{i}-m_{i-1})+1,  \text{ for } i =2, \ldots, t.$ If $h_{1}>1$, then by Theorem \ref{MistriPandey Theorem II}, the set  $A$ is an arithmetic progression.  Let $h_{1}=1$ and  $h_{2}>2$. Then we have  \[R_{1}=A \text{ and } R_{2} = S_{2} \cup T_{2}. \]  Therefore $|S_{2}|$ is minimum and hence $A_{2}=\{a_{1},a_{2},\ldots,a_{k-1}\}$ is an arithmetic progression.  Now we show that $a_{k-1}-a_{k-2}=a_{k}-a_{k-1}$.
		Let $m_{2} \leq k-3$. We have
		\begin{align*}
			a_{m_{2}+1}&<\min((h_{2}-1)^{(r)}A_{2})+a_{m_{2}+1} \\ &<\min((h_{2}-1)^{(r)}A_{2})+a_{m_{2}+2}\\  &\vdots\\ &<\min((h_{2}-1)^{(r)}A_{2})+a_{k-1}\\
			&<\min((h_{2}-1)^{(r)}A_{2})+a_{k}=\min(R_{2}).
		\end{align*}
		We also have  $R_{1}=A$ and $a_{1}<a_{2}<\cdots<a_{m_{2}+1}<a_{m_{2}+2}<\cdots<a_{k}<\min((h_{2}-1)^{(r)}A_{2})+a_{k}=\min(R_{2})$. So $\min((h_{2}-1)^{(r)}A_{2})+a_{m_{2}+i}=a_{m_{2}+i+1}$ for $i=1,2, \ldots, k-m_{2}-1$. This gives $a_{k-1}-a_{k-2}=a_{k}-a_{k-1}$.
		
		Let $m_{2}=k-2$ and $\epsilon_{2}=0$. Then 	
		\begin{align*}
			a_{k-2}&<ra_{1}+\ldots+ra_{k-2}\\ &<ra_{1}+\ldots+ra_{k-3}+(r-1)a_{k-2}+a_{k-1}\\ &<ra_{1}+\ldots+ra_{k-3}+(r-1)a_{k-2}+a_{k}=\min(R_{2}).
		\end{align*}   This  implies that $ra_{1}+\ldots+ra_{k-3}+(r-1)a_{k-2}=a_{k-1}-a_{k-2}=a_{k}-a_{k-1}$.  Let $m_{2}=k-2$ and $\epsilon_{2}\geq 1$. Then $r \geq 2$, $m_{1}=0$ and  $a_{k-1}<\min(h_{2}^{(r)}A))< \min(R_{2})$.  Note that \[\left|R_{2}\right|=2r(k-2)-\epsilon_{2}(k-3)\] and by Theorem \ref{MistriPandey Theorem I}  \[\left|h_{2}^{(r)}A\right|\geq 2r(k-2)-\epsilon_{2}(k-3)+1.\]  Let $y$ be an element of  $ h_{2}^{(r)}A$, which is different from  $\min(h_{2}^{(r)}A)$. If $y\notin R_{2}$, then \[H^{(r)}A \supseteq \{a_{1}, a_{2}, \ldots,a_{k-1}\}  \cup \{\min(h_{2}^{(r)}A),y\} \cup (\bigcup_{i=2}^{t}R_{i}).\]	This gives $\left|H^{(r)}A\right| > \sum_{i=1}^{t} \left|R_{i}\right|$, which is not possible. Therefore  $y \in R_{2}$. This gives that $ h_{2}^{(r)}A=R_{2} \cup \{\min(h_{2}^{(r)}A)\}$ and \[\left|h_{2}^{(r)}A\right|= 2r(k-2)-\epsilon_{2}(k-3)+1\]	and so by Theorem \ref{MistriPandey Theorem II}, A is an arithmetic progression.
		
		Let $h_{1}=1$ and $h_{2}=2$. Then $R_{1}=A$. Consider $R_{1}^{\prime}=\{a_{1}+a_{i}: i=2, \ldots,k-1\}$, a subset of $h_{2}^{(r)}A$. Then $\max(R_{1}^{\prime})<\min(R_{2})=a_{1}+a_{k}$. Therefore $R_{1}^{\prime}\subseteq R_{1}=A$. This gives that $a_{1}+a_{i}=a_{i+1}$ for $i=2,\ldots,k-1$.  Also $a_{k}=a_{1}+a_{k-1}<a_{2}+a_{k-1}<a_{2}+a_{k}=\min\nolimits_{+}(R_{2})$ and  $a_{k}<\min(R_{2})<\min\nolimits_{+}(R_{2})$ give $a_{2}+a_{k-1}=a_{1}+a_{k}$. Hence $A$ is an arithmetic progression.
		
		Let  $A= a_{1}+d_{1}\cdot[0,k-1]$, where $d_{1}$ is the common difference of $A$.  We show that  $H$ is an arithmetic progression with common difference $d$ and $d_{1}=da_{1}$. Note that, for all $i\in[1,t-1]$, we have \[\max\nolimits_{-}(R_{i})< \min \lbrace (h_{i+1}-h_{i})^{(r)}A_{i+1} \rbrace+ \max\nolimits_{-}(R_{i})
		< \min(R_{i+1}).\]
		But we already know that
		\[\max\nolimits_{-}(R_{i})<\max(R_{i})<\min(R_{i+1}).\]
		So
		\[	\min \lbrace (h_{i+1}-h_{i})^{(r)}A_{i+1} \rbrace+ \max\nolimits_{-}(R_{i})=\max(R_{i}).\]
		This implies that
		\begin{equation}\label{inverse thm eq I}
			\min \{ (h_{i+1}-h_{i})^{(r)}A_{i+1} \}=\max(R_{i})-\max\nolimits_{-}(R_{i}) = a_{s+1}- a_{s}=a_{2}-a_{1} \ \text{for some $s$}.
		\end{equation}
		Consider the following cases:
		\begin{enumerate}
			\item[\upshape(a)] Let $i \in [1,t-1]$ be such that $ \epsilon_{i} = \epsilon_{i+1}$. Then $m_{i+1} > m_{i}$.
			If $m_{i+1}-m_{i} \geq 2$, then $ \min\{ (h_{i+1}-h_{i})^{(r)}A_{i+1} \} = ra_{1}+ \cdots +ra_{m_{i+1}-m_{i}} > a_{2} >a_{2}-a_{1}$, which contradicts (\ref{inverse thm eq I}). Hence $ m_{i+1}-m_{i} = 1$ and $ra_{1}=(h_{i+1}-h_{i})a_{1}=a_{2}-a_{1}$.
			
			\item[\upshape(b)] Let $i \in [1,t-1]$ be such that $ \epsilon_{i} < \epsilon_{i+1}$. Then $m_{i+1} \geq m_{i}$.
			If $m_{i+1}-m_{i} \geq 1$, then $ \min\{ (h_{i+1}-h_{i})^{(r)}A_{i+1} \} = ra_{1}+ \cdots +ra_{m_{i+1}-m_{i}} + (\epsilon_{i+1}-\epsilon_{i})a_{m_{i+1}-m_{i}+1} > a_{2} >a_{2}-a_{1}$, which contradicts (\ref{inverse thm eq I}). Hence $ m_{i+1}=m_{i} $ and $(\epsilon_{i+1}-\epsilon_{i})a_{1}=(h_{i+1}-h_{i})a_{1}=a_{2}-a_{1}$.
			
			\item[\upshape(c)]  Let $i \in [1,t-1]$ be such that $\epsilon_{i}>\epsilon_{i+1} $. Then $m_{i+1} > m_{i}$.
			If $m_{i+1}>m_{i}+1$, then $ \min\lbrace (h_{i+1}-h_{i})^{(r)}A_{i+1} \rbrace=ra_{1}+ \cdots +ra_{m_{i+1}-m_{i}-1} + (r+\epsilon_{i+1}- \epsilon_{i})a_{m_{i+1}-m_{i}}> a_{2} > a_{2} - a_{1}$, which contradicts (\ref{inverse thm eq I}).
			Hence $	m_{i+1}=m_{i}+1  \text{ and } (r+\epsilon_{i+1}- \epsilon_{i})a_{1}=(h_{i+1}-h_{i})a_{1}= a_{2} - a_{1}$.
			
		\end{enumerate}
		Hence, $(h_{i+1}-h_{i})a_{1}=a_{2}-a_{1}=d_{1}$ for each  $i=1,\ldots,t-1$. So $H$ is an arithmetic progresion with common difference $d\leq r$ and $d_{1}=da_{1}$.	This completes the proof.
	\end{proof}

	\begin{cor}\label{Corollary 3.1}
		Let $r \geq 1$ and   $t > t_{0} \geq 2$   be integers. Let $A$ be a nonempty finite set of $k \geq 6$ positive integers and $H = \lbrace h_{1}, h_{2}, \ldots, h_{t} \rbrace $  be a  set  of  $t$ positive integers with $h_{1} < h_{2} < \cdots < h_{t_{0}-1}  \leq (k-1)r-1<h_{t_{0}}<\cdots<h_{t} <kr $.
		If   $(t_{0},h_{1})\not=(2,1)$ and  $	\left|H^{(r)}A\right| = \mathcal{L}(H_{t_{0}-1}^{(r)}A) +t-t_{0}+2,$
		then $H$ is  an arithmetic progression with common difference $d\leq r$ and  $A$ is an arithmetic progression with common difference $d\ast \min(A)$, where $H_{t_{0}-1}=\{h_{1},h_{2},\ldots,h_{t_{0}-1}\}$.

	\end{cor}
	\begin{proof}
		Note that
		\[\max(H_{t_{0}-1}^{(r)}A) = \max(h_{t_{0}-1}^{(r)}A)<\max\nolimits_{-}(h_{t_{0}}^{(r)}A)<\max(h_{t_{0}}^{(r)}A)<\max(h_{t_{0}+1}^{(r)}A)<\cdots<\max(h_{t}^{(r)}A)\]  and  \[H^{(r)}A \supseteq H_{t_{0}-1}^{(r)}A \cup \{\max\nolimits_{-}(h_{t_{0}}^{(r)}A)\}\bigcup \{\max (h_{i}^{(r)}A) : i=t_{0}, \ldots, t\}.\]
		Therefore $\left|H_{t_{0}-1}^{(r)}A\right|=\mathcal{L}(H_{t_{0}-1}^{(r)}A)$. If $t_{0} \geq 3$, then by Theorem \ref{Inverse Theorem},  $H_{t_{0}-1}$ is  an arithmetic progression with common difference $d$ and  $A$ is an arithmetic progression with common difference $d\ast \min(A)$. Since $(t_{0},h_{1}) \ne (2,1)$, so if $t_{0}=2$, then $h_{1}>1$. So by Theorem \ref{MistriPandey Theorem II},  $A$ is an arithmetic progression.
		
		\textbf{Claim.} If $t_{0} \geq 2$,  $t\geq t_{0}+1 $, and $A$ is an arithmetic progression with common difference $d_{1}$, then
		\begin{enumerate}
			\item $\epsilon_{t_{0}} < \epsilon_{t_{0}-1}$,
			\item $m_{t_{0}-1}=k-2$,
			\item $h_{i}-h_{i-1}=d$ for $i=t_{0}, \ldots, t$ and the common difference of $A$ is $d_{1}=da_{1}$.
		\end{enumerate}
		Now we prove our claim. Note that $m_{t_{0}-1}r + \epsilon_{t_{0}-1} = h_{t_{0}-1} \leq (k-1)r -1 = (k-2)r + r-1$. Hence $m_{t_{0}-1} \leq k-2$. Also $h_{t_{0}} \geq (k-1)r$ and $h_{t_{0}} < h_{t} \leq kr-1$, i.e., $h_{t_{0}} \leq kr-2 = (k-1)r + r-2$. Thus
		$(k-1)r \leq h_{t_{0}} \leq (k-1)r + r-2$.
		Hence $m_{t_{0}} = k-1$ and $0 \leq \epsilon_{t_{0}} \leq r-2$. Note also that
		\[\max(h_{t_{0}}^{(r)}A) = \epsilon_{t_{0}}a_{1} + ra_{2} + \cdots + ra_{k},\]
		\[\max\nolimits_{-}h_{t_{0}}^{(r)}A = (\epsilon_{t_{0}} + 1)a_{1} + (r-1)a_{2} + \cdots + ra_{k}.\]
			
		\noindent(1) If $\epsilon_{t_{0}} \geq \epsilon_{t_{0}-1}$, then
		\[\max(h_{t_{0}-1}^{(r)}A)< y = ra_{1}+\cdots + ra_{k-m_{t_{0}-1}-1}+\epsilon_{t_{0}}a_{k-m_{t_{0}-1}}+ra_{k-(m_{t_{0}-1}-1)}+\cdots+ra_{k}<\max\nolimits_{-}(h_{t_{0}}^{(r)}A), \] and $y\in h_{t_{0}}^{(r)}A $, which is a contradiction. Hence $\epsilon_{t_{0}} < \epsilon_{t_{0}-1}$. \\
		\noindent (2)  If  $m_{t_{0}-1}\leq k-3$, then
		\begin{multline*}
			\max(h_{t_{0}-1}^{(r)}A)\\<ra_{1}+\cdots + ra_{k-m_{t_{0}-1}-2}+(r-(\epsilon_{t_{0}-1}-\epsilon_{t_{0}}))a_{k-m_{t_{0}-1}-1}+\epsilon_{t_{0}-1}a_{k-m_{t_{0}-1}}+ra_{k-m_{t_{0}-1}+1}+\cdots+ra_{k}\\<\max\nolimits_{-}(h_{t_{0}}^{(r)}A),
		\end{multline*}
		which is a contradiction. Hence $\epsilon_{t_{0}} < \epsilon_{t_{0}-1}$ and  $m_{t_{0}-1}=k-2$.
		Consequently,  we can write \[\max(h_{t_{0}-1}^{(r)}A)<(r-(\epsilon_{t_{0}-1}-\epsilon_{t_{0}}))a_{1}+\epsilon_{t_{0}-1}a_{2}+ra_{3}+\cdots+ra_{k}<\max(h_{t_{0}}^{(r)}A). \]  But we already know that \[\max(h_{t_{0}-1}^{(r)}A)<\max\nolimits_{-}(h_{t_{0}}^{(r)}A)<\max(h_{t_{0}}^{(r)}A). \]
		This implies that \[(r-(\epsilon_{t_{0}-1}-\epsilon_{t_{0}}))a_{1}+\epsilon_{t_{0}-1}a_{2}+ra_{3}+\cdots+ra_{k}=\max\nolimits_{-}(h_{t_{0}}^{(r)}A), \]
		which gives $\epsilon_{t_{0}-1}=r-1$. Therefore $h_{t_{0}}-h_{t_{0}-1}=(k-1)r+\epsilon_{t_{0}}-(k-2)r-(r-1)=\epsilon_{t_{0}}+1$. Now we have
		\[\max\nolimits_{-}(h_{t_{0}-1}^{(r)}A)<(\epsilon_{t_{0}}+1)a_{1}+ ra_{2}+(r-1)a_{3}+ra_{4}+\cdots+ra_{k}<\max\nolimits_{-}(h_{t_{0}}^{(r)}A). \] We also have \[\max\nolimits_{-}(h_{t_{0}-1}^{(r)}A)<\max(h_{t_{0}-1}^{(r)}A)<\max\nolimits_{-}(h_{t_{0}}^{(r)}A). \] Therefore \[(\epsilon_{t_{0}}+1)a_{1}+ ra_{2}+(r-1)a_{3}+ra_{4}+\cdots+ra_{k}=\max(h_{t_{0}-1}^{(r)}A).\] This gives \[(\epsilon_{t_{0}}+1)a_{1}=a_{3}-a_{2}=d_{1}.\] This implies that $a_{1}$ divides $d_{1}$, so $d_{1}=da_{1}$ where $d=\epsilon_{t_{0}}+1$. Hence $ h_{t_{0}}-h_{t_{0}-1}=d$. Now we show that $h_{i}-h_{i-1}=d$ for $i=t_{0}+1, \ldots, t$.
		
		Note that $$\max\nolimits_{-}(h_{t_{0}}^{(r)}A)<\max\nolimits_{-}(h_{t_{0}+1}^{(r)}A)< \cdots <\max\nolimits_{-}(h_{t}^{(r)}A)<\max(h_{t}^{(r)}A).$$
		We already have $$\max\nolimits_{-}(h_{t_{0}}^{(r)}A)<\max(h_{t_{0}}^{(r)}A)<\max(h_{t_{0}+1}^{(r)}A) \cdots <\max(h_{t}^{(r)}A).$$ Therefore $$\max(h_{i}^{(r)}A)=\max\nolimits_{-}(h_{i+1}^{(r)}A),$$
		 which gives $(\epsilon_{i+1}-\epsilon_{i})a_{1}=a_{2}-a_{1}=d_{1}$ for $i=t_{0},t_{0}+1,\ldots,t-1$. Hence,  $H$ is an arithmetic progression with common difference $d$ and   $A$ is an arithmetic progression with common difference $da_{1}$.
		
	\end{proof}
	
	Now we discuss the case when $t=t_{0}$.

	\begin{cor}\label{Corollary 3.2}
		Let $r \geq 1$ and  $t  \geq 2$   be positive integers. Let $A$ be a nonempty finite set of $k \geq 6$ positive integers and $H = \lbrace h_{1}, h_{2}, \ldots, h_{t} \rbrace $  be a  set  of  $t$ positive integers with $h_{1} <  \cdots < h_{t-1}  \leq (k-1)r-1<h_{t} <kr $.
		If  $(t,h_{1}) \neq (2,1)$ and  $	|H^{(r)}A| = \mathcal{L}((H\setminus \{h_{t}\})^{(r)}A) +2,$
		then $H$ is  an arithmetic progression with common difference $d \leq r$ and  $A$ is an arithmetic progression with common difference $d\ast \min(A)$.

	\end{cor}
	\begin{proof}
		Note that
		\[\max((H\setminus \{h_{t}\})^{(r)}A) = \max(h_{t-1}^{(r)}A)<\max\nolimits_{-}(h_{t}^{(r)}A)<\max(h_{t}^{(r)}A)\]
		and
		\[H^{(r)}A \supseteq (H\setminus \{h_{t}\})^{(r)}A \cup \{\max\nolimits_{-}(h_{t}^{(r)}A), \max (h_{t}^{(r)}A) \}.\]
		Therefore $\left|H_{t-1}^{(r)}A\right|=\mathcal{L}((H\setminus \{h_{t}\})^{(r)}A)$. Also, if $t=2$ and $h_{1}>1$, then by Theorem \ref{MistriPandey Theorem II},  $A$ is an arithmetic progression.

		\textbf{Claim} If $t \geq 2$, then
		\begin{enumerate}
			\item $h_{t-1} > r$,
			\item $\epsilon_{t} \leq \epsilon_{t-1}$,
			\item $m_{t-1}=k-2$.
			
		\end{enumerate}
		Now we prove our claim.\\
		\noindent(1) If  $h_{t-1} \leq r$, then $\max(h^{(r)}_{t-1}A) = h_{t-1}a_{k}$. Note that \[h_{t-1}a_{k} < (\epsilon_{t}+1)a_{1} + ra_{2} + (r-1)a_{3} + \cdots + ra_{k} < (\epsilon_{t}+1)a_{1} + (r-1)a_{2} + \cdots + ra_{k} = \max\nolimits_{-}h_{t}^{(r)}A, \]
		which is a contradiction. Hence $h_{t-1} > r$ and  so $m_{t-1} \geq 1$.

		Note that $m_{t-1}r + \epsilon_{t-1} = h_{t-1} \leq (k-1)r -1 = (k-2)r + r-1$. Hence $m_{t-1} \leq k-2$. Also
		$(k-1)r \leq h_{t} \leq kr -1$.
		Hence $m_{t} = k-1$. Note also that
		\[\max(h_{t-1}^{(r)}A) = \epsilon_{t-1}a_{k-m_{t-1}} + ra_{k-m_{t-1}+1} + \cdots + ra_{k},\]
		\[\max(h_{t}^{(r)}A) = \epsilon_{t}a_{1} + ra_{2} + \cdots + ra_{k},\]
		\[\max\nolimits_{-}h_{t}^{(r)}A = (\epsilon_{t}+1)a_{1} + (r-1)a_{2} + \cdots + ra_{k}.\]
		\noindent(2) Let $\epsilon_{t} > \epsilon_{t-1}$. Then
		\begin{align*}
			\max(h_{t-1}^{(r)}A) & <x = ra_{1}+\cdots + ra_{k-m_{t-1}-1}+\epsilon_{t}a_{k-m_{t-1}}+ra_{k-(m_{t-1}-1)}+\cdots+ra_{k} \\
			& \leq y =  ra_{1} + \epsilon_{t}a_{2} + ra_{3} + \cdots + ra_{k}\\
			& \leq  (\epsilon_{t}+1)a_{1} + (r-1)a_{2} + \cdots + ra_{k} = \max\nolimits_{-}(h_{t}^{(r)}A),
		\end{align*}
		and $x, y\in h_{t}^{(r)}A $. If $x<y$ or $y< \max\nolimits_{-}(h_{t}^{(r)}A) $, then we get a contradiction. So we assume that $x = y = \max\nolimits_{-}(h_{t}^{(r)}A)$. This  implies that $\epsilon_{t} = r-1$ and $m_{t-1} = k-2$. Since $\epsilon_{t} > \epsilon_{t-1} $, we have $\epsilon_{t-1} \leq r-2$. Now consider $z = ra_{1} + ra_{2} + (r-1)a_{3} + ra_{4} + \cdots + ra_{k} \in h_{t}^{(r)}A $. Then we have $\max(h_{t-1}^{(r)}A) < z < \max\nolimits_{-}(h_{t}^{(r)}A) $, which is again a contradiction.  Hence $\epsilon_{t} \leq  \epsilon_{t-1}$. \\
		
		\noindent (3)  If  $m_{t-1}\leq k-3$, then
		\begin{multline*}
			\max(h_{t-1}^{(r)}A)\\<ra_{1}+\cdots + ra_{k-m_{t-1}-2}+(r-(\epsilon_{t-1}-\epsilon_{t}))a_{k-m_{t-1}-1}+\epsilon_{t-1}a_{k-m_{t-1}}+ra_{k-m_{t-1}+1}+\cdots+ra_{k}\\<\max\nolimits_{-}(h_{t}^{(r)}A),
		\end{multline*}
		which is a contradiction. Hence $\epsilon_{t} \leq \epsilon_{t-1}$ and  $m_{t-1}=k-2$.
		Consequently,  we can write \[\max(h_{t-1}^{(r)}A)<(r-(\epsilon_{t-1}-\epsilon_{t}))a_{1}+\epsilon_{t-1}a_{2}+ra_{3}+\cdots+ra_{k}<\max(h_{t}^{(r)}A). \]  But we already know \[\max(h_{t-1}^{(r)}A)<\max\nolimits_{-}(h_{t}^{(r)}A)<\max(h_{t}^{(r)}A). \] This implies \[(r-(\epsilon_{t-1}-\epsilon_{t}))a_{1}+\epsilon_{t-1}a_{2}+ra_{3}+\cdots+ra_{k}=\max\nolimits_{-}(h_{t}^{(r)}A). \] This gives $\epsilon_{t-1}=r-1$. Therefore $h_{t}-h_{t-1}=(k-1)r+\epsilon_{t}-(k-2)r-(r-1)=\epsilon_{t}+1$.
		 We have
		\[\max\nolimits_{-}(h_{t-1}^{(r)}A)<(\epsilon_{t}+1)a_{1}+ ra_{2}+(r-1)a_{3}+ra_{4}+\cdots+ra_{k}<\max\nolimits_{-}(h_{t}^{(r)}A). \]
		We also have \[\max\nolimits_{-}(h_{t-1}^{(r)}A)<\max(h_{t-1}^{(r)}A)<\max\nolimits_{-}(h_{t}^{(r)}A). \] Therefore \[(\epsilon_{t}+1)a_{1}+ ra_{2}+(r-1)a_{3}+ra_{4}+\cdots+ra_{k}=\max(h_{t-1}^{(r)}A).\] This gives
		\begin{equation}\label{Eq- x}
			(\epsilon_{t}+1)a_{1}=a_{3}-a_{2}.
		\end{equation}
		If $t \geq 3$, by Theorem \ref{Inverse Theorem},  $H\setminus \{h_{t}\}$ is  an arithmetic progression with common difference $d \leq r$ and  $A$ is an arithmetic progression with common difference $d\ast \min(A)$. Therefore
		\[(\epsilon_{t}+1)a_{1}=a_{3}-a_{2} = da_{1},\] which implies $h_{t} - h_{t-1} = \epsilon_{2} +1 = d$. Hence, if $t\geq 3$, $H$ is  an arithmetic progression with common difference $d \leq r$ and  $A$ is an arithmetic progression with common difference $d\ast \min(A)$.  If $t=2$, then $H = \{h_{1}, h_{2}\}$ is an arithmetic progression with common difference $d= h_{2}-h_{1}=\epsilon_{t} + 1 \leq r$. Since $h_{1}>1 $  and  $\left|H_{1}^{(r)}A\right|=\mathcal{L}((H\setminus \{h_{2}\})^{(r)}A) = |h_{1}^{(r)}A| = m_{1}r(k-m_{1}) + \epsilon_{1}(k-2m_{1}-1) + 1 $, we have from Theorem \ref{MistriPandey Theorem II} that  $A$ is an arithmetic progression with   common difference  $da_{1}$ from  (\ref{Eq- x}).

	\end{proof}

	\begin{cor}\label{Corollary 3.3}
		Let $r \geq 2$ be a positive integer and  $A$ be a nonempty finite set of $k \geq 6$ positive integers and $H$  be a  set  of  $t\geq 2$ positive integers with $ (k-1)r-1<\min(H)<\max(H) <kr $. Let $m_{1}=\lfloor \min(H)/r \rfloor$.
		If   \[|H^{(r)}A| = m_{1}r(k-m_{1})+(h_{1}-m_{1}r)(k-2m_{1}-1)+t,\]
		then $H$ is  an arithmetic progression with common difference $d \leq r-1$ and  $A$ is an arithmetic progression with common difference $d\ast \min(A)$.
	\end{cor}
	\begin{proof}
		Note that \[\max(h_{2}^{(r)}A)<\max(h_{3}^{(r)}A)<\cdots<\max(h_{t}^{(r)}A)\] and
		\[H^{(r)}A \supseteq h_{1}^{(r)}A \cup \{\max(h_{i}^{(r)}A) : 2\leq i \leq t\}.\] Therefore $\left|h_{1}^{(r)}A\right|=m_{1}r(k-m_{1})+(h_{1}-m_{1}r)(k-2m_{1}-1)+1$ and by  Theorem \ref{MistriPandey Theorem II},  $A$ is an arithmetic progression. Assume $d_{1}$ is the common difference of $A$.
		Note that
		\[\max\nolimits_{-}(h_{1}^{(r)}A)<\max\nolimits_{-}(h_{2}^{(r)}A)< \cdots <\max\nolimits_{-}(h_{t}^{(r)}A)<\max(h_{t}^{(r)}A).\]
		We already have
		\[\max\nolimits_{-}(h_{1}^{(r)}A)<\max(h_{1}^{(r)}A)<\max(h_{2}^{(r)}A) \cdots <\max(h_{t}^{(r)}A).\]
		Therefore
		\[\max(h_{i}^{(r)}A)=\max\nolimits_{-}(h_{i+1}^{(r)}A),\]
		which gives $(\epsilon_{i+1}-\epsilon_{i})a_{1}=a_{2}-a_{1}=d_{1}$ for $i=1,2,\ldots,t-1$. Hence, set $H$ is an arithmetic progression with common difference $d \leq r-1$ and  set $A$ is an arithmetic progression with common difference $d\ast \min(A)$.
	\end{proof}

	\begin{cor}\label{Corollary 3.4}
		Let  $r$ be a positive  integer, $A$ be a  finite set of $k\geq 7$ nonnegative  integers with $0 \in A$, and $H$  be a  set  of  $t \geq 2$ positive integers with $1 \leq r \leq  \max(H) \leq (k-2)r-1 $. Let $ m = \lceil \min(H)/r \rceil$ and $ m_{1}=\lfloor \min(H)/r  \rfloor$.  If
		\begin{equation}\label{Cor 3.4-Eq-1}
			|H^{(r)}A| =    m_{1}r(m-m_{1}+1) + (\min(H)-m_{1}r)(m-2m_{1}) + \mathcal{L}\Big(H^{(r)}(A\setminus\{0\}) \Big),
		\end{equation}
		then  $H$ is an arithmetic progression with common difference $d \leq r$ and   $A$ is an arithmetic progression with common difference $d\ast\min(A\setminus\{0\})$.  Moreover, if $\min(H)>1$, then $d=1$.
		
	\end{cor}
	
	\begin{proof}
		Let $A=\{0,a_{1},\ldots,a_{k-1}\}$ be a set of nonnegative integers with $0<a_{1}< \cdots< a_{k-1}$ and $H = \{h_{1}, h_{2}, \ldots, h_{t}\}$ be set of positive integer  such that $h_{1} < h_{2} < \cdots < h_{t}$ .  From (\ref{Cor 3.4-Eq-1}) and  Corollary \ref{Corollary 2.1}, we have \[\left|h^{(r)}_{1}B\right|= m_{1}r(m-m_{1}+1) + (h_{1}-m_{1}r)(m-2m_{1})+1\] and \[\left|H^{(r)}A^{\prime} \right| =   \mathcal{L}\Big(H^{(r)}(A^{\prime}) \Big)  , \]
		where $A^{\prime}=\{a_{1},a_{2},\ldots,a_{k-1}\}$ and	 $ B=\{0,a_{1}, \ldots, a_{m}\}$ with  $m=\lceil h_{1}/r \rceil$. Then by Theorem \ref{Inverse Theorem},  $H$ is  an arithmetic progression with common difference $d \leq r$ and  $A^{\prime}$ is an arithmetic progression with common difference $d\ast \min(A^{\prime})$. Now, we show that $d=1$, if $h_{1}>1$. To show $d=1$, it is sufficient to prove that common difference of arithmetic progression $A$ is $a_{1}$. If  $r=1$, then $d=1$. Assume  $r\geq 2$. Now, define $R_{i} =S_{i} \cup T_{i}$  for the set $A^{\prime}$ as it was  defined in Theorem \ref{Direct Theorem}. So
		\begin{equation*}
			R_{1}=S_{1}=h^{(r)}_{1}A^{\prime} \subseteq h^{(r)}_{1}A.
		\end{equation*}
		Now $\max(h^{(r)}_{1}A^{\prime})=\max(h^{(r)}_{1}A)$ implies that  $h^{(r)}_{1}A \cap R_{2}=\emptyset$.  We write
		\begin{multline}\label{bound}
			\left|H^{(r)}A\right|
			=  m_{1}r(m-m_{1}+1) + (h_{1}-m_{1}r)(m-2m_{1})+ \left|h_{1}^{(r)}A^{\prime}\right|  \\ + \sum_{i=2}^{t} \Big( r(m_{i}-m_{i-1})(k-m_{i}-1)+(\epsilon_{i}-\epsilon_{i-1})(k-m_{i}-2) \\   -(\max\{\epsilon_{i},\epsilon_{i-1}\})(m_{i}-m_{i-1})+1\Big).
		\end{multline}
		$\clubsuit$ If $h_{1} = m_{1}r + \epsilon_{1}$ with $m_{1} \geq 1$ and $\epsilon_{1} \geq 1$, then $m = m_{1} + 1$ and so $\left|B\right| \geq 3$. Since $\left|h^{(r)}_{1}B\right|= m_{1}r(m-m_{1}+1) + (h_{1}-m_{1}r)(m-2m_{1})+1$, therefore  Theorem \ref{MistriPandey Theorem II} implies that $B$ is an arithmetic progression with common difference $a_{1}$ and we know that $A^{\prime}$ is also an arithmetic progression. Hence $A = B \cup A^{\prime}$ is an arithmetic progression with common difference  $a_{1}$.\\
		$\clubsuit$ If $h_{1} = m_{1}r$, then $m_{1} = m$, and if  $h_{1} < r$, then $m_{1} = 0$ and $m =1$. Since $h^{(r)}_{1}B \cup h_{1}^{(r)}A^{\prime} \subseteq h_{1}^{(r)}A $, so from (\ref{bound}), we have 
		\begin{align*}
			\left|h_{1}^{(r)}A \right| &= \left|h^{(r)}_{1}B \right| + \left| h_{1}^{(r)}A^{\prime} \right|-1\\ &=m_{1}r(m-m_{1}+1) + (h_{1}-m_{1}r)(m-2m_{1})+m_{1}r(k-m_{1}-1)\\ &~~~~~~~~~~~~~~~~+(h_{1}-m_{1}r)(k-2m_{1}-2)+1\\
			&\leq m_{1}r+(h_{1}-m_{1}r)+	m_{1}r(k-m_{1}-1)+(h_{1}-m_{1}r)(k-2m_{1}-2)+1\\ &=m_{1}r(k-m_{1})+(h_{1}-m_{1}r)(k-2m_{1}-1)+1.
		\end{align*}
		This gives
		\[\left|h_{1}^{(r)}A\right|=  m_{1}r(k-m_{1})+(h_{1}-m_{1}r)(k-m_{1}-1)+1. \]
		So, by Theorem \ref{MistriPandey Theorem II}, $A$ is an arithmetic progression with common difference $a_{1}$.  This implies that $d=1$. Hence, \[H=h_{1}+[0,t-1] \text{ and } A= \min(A\setminus \{0\})\ast [0,k-1].\]
		
	\end{proof}
	As a consequence of Corollary \ref{Corollary 3.1}, Corollary \ref{Corollary 3.2}, Corollary \ref{Corollary 3.3}  and Corollary \ref{Corollary 3.4}, we have the following Corollaries.
	\begin{cor}\label{Corollary 3.5}
		Let $r \geq 1$ and  $t > t_{0} \geq 2$   be integers. Let $A$ be a nonempty finite set of $k \geq 7$ nonnegative integers with $0 \in A$ and $H = \lbrace h_{1}, h_{2}, \ldots, h_{t} \rbrace $  be a  set  of  $t$ positive integers with $h_{1} < h_{2} < \cdots < h_{t_{0}-1}  \leq (k-2)r-1<h_{t_{0}}<\cdots<h_{t} <(k-1)r $.
		If   $(t_{0},h_{1})\not=(2,1)$ and  $\left|H^{(r)}A\right| \geq    m_{1}r(m-m_{1}+1) + (\min(H)-m_{1}r)(m-2m_{1}) + \mathcal{L}\Big(H^{(r)}(A\setminus\{0\}) \Big) +  t-t_{0}+2$,
		then   $H$ is an arithmetic progression with common difference $d \leq r$ and  $A$ is an arithmetic progression with common difference $d\ast\min(A\setminus\{0\})$.  Moreover, if $\min(H)>1$, then $d=1$.
	\end{cor}

	\begin{cor}\label{Corollary 3.6}
		Let $r \geq 1$ and  $t  \geq 2$   be integers. Let $A$ be a nonempty finite set of $k \geq 7$ nonnegative integers with $0 \in A$ and $H = \lbrace h_{1}, h_{2}, \ldots, h_{t} \rbrace $  be a  set  of  $t$ positive integers with $h_{1} < h_{2} < \cdots < h_{t-1}  \leq (k-2)r-1<h_{t} <(k-1)r $.
		If $(t,h_{1}) \neq (2,1)$ and     $\left|H^{(r)}A\right| \geq    m_{1}r(m-m_{1}+1) + (\min(H)-m_{1}r)(m-2m_{1}) + \mathcal{L}\Big(H^{(r)}(A\setminus\{0\}) \Big) +2$,
		then   $H$ is an arithmetic progression with common difference $d \leq r$ and   $A$ is an arithmetic progression with common difference $d\ast\min(A\setminus\{0\})$.  Moreover, if $\min(H)>1$, then $d=1$.
	\end{cor}
	
	\begin{cor}\label{Corollary 3.7}
		Let $r \geq 2$ be a positive integer and  $A$ be a nonempty finite set of $k \geq 7$ nonnegative integers with $0 \in A$ and $H$  be a  set  of  $t\geq 2$ positive integers with $ (k-2)r-1<\min(H)<\max(H) <(k-1)r $. Let $m_{1}=\lfloor \min(H)/r\rfloor$.
		If   \[|H^{(r)}A| = m_{1}r(k-m_{1})+(\min(H)-m_{1}r)(k-2m_{1}-1)+t,\]
		then  $H$ is an arithmetic progression with common difference $d \leq r-1$ and  $A$ is an arithmetic progression with common difference $d\ast\min(A\setminus\{0\})$.
	\end{cor}
	
	\section{Conclusions}
	In \S 1.1, we have already discussed the relation between generalized $H$-fold sumset and  subsequence sum. Choosing a particular $H$ we get some results of subsequence sum.
	
	\begin{cor}
		
		\textup{\cite[Theorem 2.1]{MISTRIPANDEYPRAKASH2015}}\label{MistriPandeyOM Theorem I}
		Let $k$  and $r$ be positive integers. Let $\mathbb{A}$ be a finite sequence of nonnegative integers with $k$ distinct terms each with repetitions $r$.\\
		If $0 \notin \mathbb{A}$ and $k \geq 3$, then
		\begin{equation*}
			\sum(\mathbb{A}) \geq \dfrac{rk(k+1)}{2}.
		\end{equation*}
		If $0\in \mathbb{A}$ and $k \geq 4$, then
		\begin{equation*}
			\sum(\mathbb{A}) \geq 1 + \dfrac{rk(k-1)}{2}.
		\end{equation*}
		The above lower bounds are best possible.
	\end{cor}
	\begin{proof}
		If $0 \notin \mathbb{A}$, then taking $H = [1,kr]$ in Remark \ref{Remark 2.2} (b), we get $	\sum(\mathbb{A}) \geq \dfrac{rk(k+1)}{2}.$ If $0 \in \mathbb{A}$, then taking $H=[1,kr]$ in Remark \ref{Remark 2.3} (a), we get $	\sum(\mathbb{A}) \geq 1 + \dfrac{rk(k-1)}{2}.$
	\end{proof}
	
	\begin{cor}\textup{\cite[Theorem 2.3]{MISTRIPANDEYPRAKASH2015}}\label{MistriPandeyOM Theorem II}
		Let $k$  and $r$ be positive integers. If $\mathbb{A}$ is a finite sequence of nonnegative  integers with $k$ distinct terms each with repetitions $r$. \\
		If $0 \notin \mathbb{A}$,  $k \geq 6$  and
		\begin{equation*} \label{MPO Eqn 3}
			\sum(\mathbb{A}) = \dfrac{rk(k+1)}{2},
		\end{equation*}
		then $\mathbb{A} = d \ast [1,k]_{r}$ for some positive integer $d$.\\
		If $0\in \mathbb{A}$, $k \geq 7$ and
		\begin{equation*}\label{MPO Eqn 4}
			\sum(\mathbb{A}) = 1 + \dfrac{rk(k-1)}{2},
		\end{equation*}
		then $\mathbb{A} = d \ast [0,k-1]_{r}$ for some positive integer $d$.
	\end{cor}
	\begin{proof}
		If $0 \notin \mathbb{A}$, then taking $H = [1,kr]$ in Corollary \ref{Corollary 3.1}, we get $\mathbb{A} = d \ast [1,k]_{r}$ for some positive integer $d$. If $0 \in \mathbb{A}$, then taking $H=[1,kr]$ in Corollary \ref{Corollary 3.5}, we get $\mathbb{A} = d \ast [0,k-1]_{r}$ for some positive integer $d$.
	\end{proof}
	Taking $H = [\alpha,kr]$ in Theorem \ref{Direct Theorem} and Remark \ref{Remark 2.2} and taking $H = [\alpha,(k-1)r]$ in Corollary \ref{Corollary 2.1} and Remark \ref{Remark 2.3},   we get the following result.
	\begin{cor}
		\textup{\cite[Corollary 3.2]{BHANJA2020}}\label{BHANJA2020 I}
		Let $k\geq 4$,  $r \geq 1$ and $\alpha$ be  integers with $1 \leq \alpha < kr$. Let $m \in [1, k]$ be the integer such that $(m - 1)r \leq  \alpha < mr$. Let $\mathbb{A}$ be a finite sequence of nonnegative integers with $k$ distinct terms each with repetitions $r$.\\
		If $0 \notin \mathbb{A}$, then
		\begin{equation*}
			\sum\limits_{\alpha}(\mathbb{A}) \geq \dfrac{rk(k+1)}{2} - \dfrac{rm(m+1)}{2} + m(mr-\alpha) + 1.
		\end{equation*}
		If $0\in \mathbb{A}$, then
		\begin{equation*}
			\sum\limits_{\alpha}(\mathbb{A}) \geq \dfrac{rk(k-1)}{2} - \dfrac{rm(m-1)}{2} + (m-1)(mr-\alpha) + 1.
		\end{equation*}
		The above lower bounds  are best possible.
	\end{cor}
	
	Taking $H = [\alpha,kr-2]$, Theorem \ref{Inverse Theorem} and Corollaries \ref{Corollary 3.1} - \ref{Corollary 3.7} give the following result.
	\begin{cor}
		\textup{\cite[Corollary 3.5]{BHANJA2020}}\label{BHANJA2020 I}
		Let $k\geq 7$,  $r \geq 1$ and $\alpha$ be  integers with $1 \leq \alpha \leq kr-2$. Let $m \in [1, k]$ be the integer such that $(m - 1)r \leq  \alpha < mr$. Let $\mathbb{A}$ be a finite sequence of nonnegative integers with $k$ distinct terms each with repetitions $r$.\\
		If $0 \notin \mathbb{A}$ and
		\begin{equation*}
			\sum\limits_{\alpha}(\mathbb{A}) = \dfrac{rk(k+1)}{2} - \dfrac{rm(m+1)}{2} + m(mr-\alpha) + 1,
		\end{equation*}
		then $\mathbb{A} = d \ast [1,k]_{r}$ for some positive integer $d$.\\
		If $0\in \mathbb{A}$ and
		\begin{equation*}
			\sum\limits_{\alpha}(\mathbb{A}) = \dfrac{rk(k-1)}{2} - \dfrac{rm(m-1)}{2} + (m-1)(mr-\alpha) + 1,
		\end{equation*}
		then $\mathbb{A} = d \ast [0,k-1]_{r}$ for some positive integer $d$.
	\end{cor}

	\section*{Acknowledgment}
	
	The first author would like to thank to the Council of Scientific and Industrial Research (CSIR), India for providing the grant to carry out the research with Grant No. 09/143(0925)/2018-EMR-I.

	\bibliographystyle{amsplain}

\end{document}